\documentclass[11pt,reqno]{amsart}
\usepackage[left=33mm,right=33mm,top=30mm,bottom=32mm]{geometry}
\usepackage{mathtools,amssymb,amsthm,mathrsfs,color,float}
\usepackage{paralist}
\usepackage{stackengine}
\usepackage{centernot}
\usepackage{txfonts}
\usepackage{appendix}
\usepackage{soul}
\sethlcolor{yellow}
\usepackage{tikz}
\usepackage{bm}
\usepackage[colorlinks,
linkcolor=red,
anchorcolor=green,
citecolor=blue, 
]{hyperref}

\usepackage[T1]{fontenc}
\usepackage[utf8]{inputenc}

\usepackage{calc}
\definecolor{bleu1}{RGB}{0,57,128}
\def\bleu1{\color{bleu1}}

\usepackage{etoolbox}
\patchcmd{\section}{\normalfont}{\normalfont \bleu1}{}{}
\patchcmd{\subsection}{\normalfont}{\normalfont \bleu1}{}{}
\patchcmd{\subsubsection}{\normalfont}{\normalfont \bleu1}{}{}

\newtheorem{The}{\bleu1 Theorem}[section]
\newtheorem{Pro}[The]{\bleu1 Proposition}
\newtheorem{Lem}[The]{\bleu1 Lemma}

\newtheorem{Con}[The]{\bleu1 Conjecture}
\theoremstyle{definition}
\newtheorem{defn}[The]{\bleu1 Definition}

\newtheorem{Rem}[The]{\bleu1 Remark}

\numberwithin{equation}{section}

\newcommand{\T}{\mathbb{T}}
\newcommand{\R}{\mathbb{R}}
\newcommand{\Z}{\mathbb{Z}}
\newcommand{\N}{\mathbb{N}}

\newcommand{\Q}{\mathbb{Q}}

\newcommand{\Lip}{\mathrm{Lip\,}}

\newcommand{\Var}{\mathrm{Var}\,}
\newcommand{\Cov}{\mathrm{Cov}\,}
\newcommand{\supp}{\mathrm{supp\,}}

\makeatletter
\@namedef{subjclassname@2020}{\textup{2020} Mathematics Subject Classification}
\makeatother

\title[On Aubry's completeness conjecture]{On Aubry's completeness conjecture}
\author{ Tianqi Shi \and Jinxin Xue}
\address[Tianqi Shi]{School of Mathematics and Statistics, Northeast Normal University, Changchun 130024, China}
\email{tqshi.math@gmail.com}
\address[Jinxin Xue]{New Cornerstone Sciences Laboratary at Department of Mathematical Sciences, Tsinghua University, Beijing 100084, China}
\email{jxue@tsinghua.edu.cn}
\keywords{Aubry's completeness conjecture, devil's staircase, Frankel-Kontorova model, twist map, hyperbolicity, KAM torus}
\subjclass[2020]{37J40,37J51}
\begin{document}
\maketitle
\begin{abstract}
In this paper, we prove Aubry's completeness stating conjecture that for a twist map the graph of rotation numbers as a function of the cohomology classes is a purely singularly continuous function (called complete devil's  staircase by Aubry) when the set of all minimal configurations is uniformly hyperbolic. Such a phenomenon is crucial for characterizing the chain of atoms being an insulator for the Frenkel-Kontorova model, and can be considered as the analogue of the phase locking phenomenon in critical circle maps as well as the fractional quantum Hall effect. In contrast, in the presence of a positive measure set of KAM tori, we prove that the devil's staircase is incomplete. 
\end{abstract}
\tableofcontents
\section{Introduction}
The Aubry-Mather theory on twist maps is widely considered to be the most important development in Hamiltonian dynamics after KAM. An important prototypical model studied in the theory is the \textit{standard map} of the form $$\phi: \T\times \R\to \T\times \R,\quad (x,y)\mapsto(x+y+k \sin x,y+k \sin x).$$ 
 The standard map is in general very chaotic, defying analysis. When $k$ is small, the theory of Kolmogorov-Arnold-Moser gives a Cantor set of Diophantine numbers such that for each point $\rho$ in the set, there is an invariant circle restricted to which the dynamics is conjugate to a circle rotation by $\rho.$ However, the KAM theory says nothing in the complement of the set of KAM invariant circles where the dynamics is very chaotic. 
The approach of Aubry-Mather theory was to utilize the variational principle to obtain a family of minimal configurations, which can  be classified, and we can extract an ordered structure from the chaos by the energy minimization procedure. Compared to the KAM theory, the Aubry-Mather theory is nonperturbative. Moreover, it has a global nature lying in the fact that we can parameterize all the minimal configurations by a real parameter, and the dynamical information of the minimal configurations is encoded in the structure of a function called  the \textit{$\alpha$-function} of the real parameter. The main theme of the paper is to further explore how the structure of $\alpha$ is related to the phase space dynamics.

\subsection{Frenkel-Kontorova model, devil's  staircase and Aubry's completeness conjecture}
To motivate Aubry's completeness conjecture, we first give a brief overview of Aubry's work on this topic. Readers are recommended the paper \cite{aubry1978new} of Aubry in 1978, which set the stage for his series of works later. 

\subsubsection{Frenkel-Kontorova model and minimal configurations}\
Aubry focused on the \emph{Frenkel-Kontorova model} (FK model) in solid state physics 
which describes the positions of atoms in solids with nearest neighbor interaction with energy of the form
 \index{Frenkel-Kontorova model}
 \begin{equation}\label{EqFK}
h(x_n,x_{n+1})=\frac{1}{2}(x_n-x_{n+1})^2-k\cos (2\pi x_n), 
 \end{equation}
where the term $(x_n-x_{n+1})^2$ is similar to the elastic potential for springs and the term $k\cos(2\pi x_n)$ is the periodic background potential.       A configuration is a biinfinite sequence $(\ldots,x_{-1}, x_0,x_1,\ldots)$. When the atoms are at a stationary configuration, we should have the Euler-Lagrange equation 
\begin{equation}\label{EqELTwistMap}
\partial_2h(x_{i-1},x_i)+\partial_1h(x_i,x_{i+1})=0.    
\end{equation}
 Moreover, in experiments, we only see stable stationary configurations, since unstable ones are easily perturbed away. It is then natural to look for stable configurations by minimizing the total energy.  In general, if we consider the total energy $\sum_{\ell=-\infty}^\infty h(x_\ell,x_{\ell+1})$ of a configuration, it is most likely to be infinity.  However, we can still talk about minimality in terms of minimality for every finite segment. 
    \begin{defn}\index{minimal configurations}
	A configuration is called \emph{minimal},  if the following holds for any integers $i<j$
	$$
	\sum_{\ell=i}^{j-1}h(x_\ell,x_{\ell+1})\leqslant\inf_{\xi_i=x_i,\xi_j=x_j} \sum_{\ell=i}^{j-1}h(\xi_\ell,\xi_{\ell+1}).
	$$
\end{defn}

In the FK model, the  Euler-Lagrange equation gives 
 \begin{equation}\label{EqELFK}
(x_{n+1}-2x_n+x_{n-1})-2\pi k\sin (2\pi x_n)=0.     
 \end{equation}
 Defining $y_n=\partial_{2}h(x_{n-1},x_{n})= x_{n}-x_{n-1}$, we get a 2D map 
$$\left\{
\begin{array}{l}
	x_{n+1}=x_n+y_n+2\pi k\sin(2\pi x_n)\\
    y_{n+1}=y_n+2\pi k\sin(2\pi x_n).\\
\end{array}
\right.$$
This is in the form of the standard map.

\subsubsection{The devil's  staircase}
    Aubry in 1983 (\cite{MR719055}) introduced the \textit{mean energy per atom} along any minimal configuration with the prescribed rotation number,
$$
\beta(\varrho)=\lim_{k\to\infty}\frac 1{2k+1}\sum_{i=-k}^{k-1}h(x_i,x_{i+1}),\ \forall \ (\cdots,x_i,x_{i+1},\cdots)\in\mathscr Q_\varrho.
$$
Here $\mathscr Q_\varrho$ is the set of minimal configurations with rotation number $\varrho$. Aubry first proved the strict convexity of $\beta$ in \cite{MR719055}. This is essentially reflecting the fact that minimal configurations with different rotation numbers necessarily intersect. 

 In the paper \cite{MR719055}, Aubry discovered the following remarkable devil's  staircase phenomenon encoding further dynamical information of  the set of minimal configurations for all rotation numbers. 
\begin{The}\label{ThmDevil}
The $\beta$-function is differentiable at each irrational point and generically nondifferentiable at each rational point.    
\end{The}
 The result was discovered by Aubry using numerics who formulated it into a conjecture. The mathematical proof of the result was given in \cite{MR967638,MR1384394}. The result can be reformulated in terms of the Legendre dual of the $\beta$-function called the $\alpha$-function. The notion was also introduced by Aubry in \cite{MR719055} who considered adding a chemical potential to the FK model. 
\begin{defn}
    The \textit{$\alpha$-function} is defined to be the Legendre transformation of $\beta$, i.e. the $\alpha$-function takes its value at $c$ such that $\alpha(c)+\beta(\varrho)=c\varrho$, or equivalently,
$$
\alpha(c)=-\liminf_{K,K'\to\infty}\inf_{\mathbf{x}\in\mathbb{R}^{K'+K}}
\frac{1}{K+K'}\Big(\sum_{i=-K}^{K'-1}h(x_i,x_{i+1})-c(x_{K'}-x_{-K})\Big).
$$
\end{defn}

The last theorem is translated to the statement that the graph of $D\alpha$ is a \textit{devil's staircase}, where by devil's staircase we mean that the graph of $D\alpha$ is monotonically increasing and has infinitely many horizontal pieces. Indeed, since $\alpha$ and $\beta$ are mutual Legendre transformations of each other, the convexity of $\beta$ implies also the convexity of $\alpha$, which further implies that $D\alpha$ is monotonically increasing.   the strict convexity of $\beta$ translates to the continuity of $D\alpha. $ A nondifferentiable point $\varrho$ of $\beta$ corresponds to an interval $[c^-,c^+]$ such that $D\alpha=\varrho$ and $D^2\alpha=0$ when restricted to $(c^-,c^+)$. Thus, generically, the graph of $D\alpha$ has infinitely many horizontal pieces, i.e., a devil's staircase.

\subsubsection{Transition by breaking of analyticity and the completeness  conjecture}
There is an important theme in Aubry's works that seems not well understood or appreciated by mathematicians that is the {\it transition by breaking of analyticity}. The physics of this transition is important. On an invariant curve, the atoms can slide without causing extra energy, thus the chain of atom is a metal. On the other hand, when the invariant curve is broken,  some energy (called the \textit{Peierls-Nabarro barrier}) has to be costed to shift the minimal configuration to a nearby one, in which case, the atoms are locked and the material is an insulator. When we increase the constant $k$ in the standard map from zero, for a Diophantine number $\rho$, the corresponding invariant circle persists for $k$ in a certain range $[0,k_\rho]$. When $k$ further increases beyond $k_\rho$, the invariant circle breaks and a Cantor set shows up. This is a transition from metal to insulator called the transition by breaking of analyticity by Aubry. 

In physics, it is important to know if the devil's staircase is complete or not. We say a function $f:\ [a,b]\to \R$ is a \textit{complete} devil's staircase if $f$ is a purely singularly continuous function, i.e. $f$ is continuous, nonconstant, and its derivative is vanishing almost everywhere.   A complete devil's staircase means zero probability for incommensurate phases. As a commensurate phase means that the atoms are locked, thus the material is an insulator, a complete devil's staircase means that the material is an insulator with probability 1. This is called global hysteresis by Aubry. 



An important conjecture left unsolved in the work of Aubry is the following conjecture (\cite{MR1219348}). Considering Aubry's series of works starting from \cite{aubry1978new}, it is clear that the conjecture is one of the central concerns in his theory. 

\begin{Con}Let $\Lambda_{[c_1,c_2]}$ be the set of orbits of a twist map corresponding to minimal configurations with  $c\in [c_1,c_2]$. If $\Lambda_{[c_1,c_2]}$ is uniformly hyperbolic, then the graph of $D\alpha:\ [c_1,c_2]\to \R$ is purely singularly continuous.  
\end{Con}


Aubry proved the conjecture for the FK model when $k$ is greater than a precisely given constant (\cite{MR1219348}). Aubry also gave several explicitly computed models exhibiting complete devil's staircases (\cite{Aubry_1983exact,aubry1983complete}).

The hyperbolicity assumption is indeed a rather sharp condition in the sense that it is not only sufficient but also necessary. It was proved by Aubry, MacKay and Baesens  that for minimal configurations, the uniform hyperbolicity is equivalent to the nondegeneracy of the second variation of the action functional. The latter further implies that some nonvanishing energy has to be paid to shift a minimal configuration to a nearby one (positive Peierls-Nabarro barrier), in which case the material is an insulator. 

The mathematical interest of the conjecture is self-evident, since the global functions $\alpha$ and $\beta$ encode the information of the phase space dynamics. Whereas the strict convexity of $\beta$ reflects  the fact that minimal configurations with different rotation number intersects, the nondifferentiability of $\beta$ at rational points reflects the gaps between nearby minimal configurations, Theorem \ref{ThmDevil} means that the hyperbolicity of the minimal configurations is encoded in the completeness of the devil's staircase. Whereas the strict convexity of $\beta$ involves the information of the graph of $\beta$ and Theorem \ref{ThmDevil} involves the differentiability of $\beta$, the Aubry's completeness conjecture involves finer information of $\beta$ concerning second order derivatives. 

\subsubsection{Complete devil's staircase in physics and mathematics}
The completeness of devil's staircase turns out to be an important theme in condensed matter physics. \textit{Fractional quantum Hall effect} (FQHE) is a phenomenon discovered in 1982 by Tsui, St\"ormer and Gossard  in certain strongly correlated quantum systems that the quantum number could take on rational fractional values. In the paper \cite{rotondo2016devil}, the authors considered the thin-torus limit of the Hamiltonian of interacting electrons in a strong magnetic field and revealed the devil's staircase structure of the inverse filling fraction versus the strength of magnetic field. The model studied by Aubry was the zero-temperature limit of this model. This is an active field of researches in physics. We refer readers to \cite{MR668663,MR675056,ishimura1985devil,abe2024large,ruzzi2020hidden} for other related researches.  It is important to mention that there is a remarkable parallel between the Aubry's theory on twist maps and the Hubbard model in cold atom physics (\cite{MR710121}). In both settings, the rational ground states are classified and there is a devil's staircase structure. It would be very interesting to further explore the analogue. 

For circle maps, there is a similar phenomenon called \textit{phase locking}. Consider circle map of the form $x\mapsto x+\alpha+\frac{\epsilon}{2\pi}\sin(2\pi x)$. When $\epsilon=1$, treating the rotation number $\rho$ as a function of $\alpha$, we get that for full measure set of $\alpha$, the rotation number is rational. The graph of $\rho$ as a function of $\alpha$ is a complete devil's staircase (see \cite{MR1374412}). The phase locking is caused by the existence of a critical point of the map. The mechanism for the devil's staircase in this case is rather different from that of our case. 

\subsection{The main theorems}
In this paper, we confirm the conjecture for generic Tonelli Lagrangian systems. According to Moser's theorem (\cite{MR863203}, see also Theorem \ref{thm:moser}), a twist map can be equivalently written as the time-1 map of a Tonelli Lagrangian system. So for convenience of notations, we work in the framework of Tonelli Lagrangian systems. 
To set up the problem, we consider a Tonelli Lagrangian $L:\T\times T\T\to \R$ that is a $C^2$ function $L(t,x,v)$ that is strictly convex and superlinear $(\frac{L}{|v|}\to\infty$ as $|v|\to\infty$) in component $v$ and has complete Euler-Lagrangian flows (see Section \ref{subsec:cvham}). 

We introduce the $A(\mu):=\int Ld\mu$ for an invariant probability measure $\mu$ defined on $\T\times T\T$. Let $\eta$ be a closed 1-form on $\T^1$ with cohomology class $c\in H^1(\T,\R)$. We introduce the \textit{$\alpha$-function} 
$$\alpha(c):=-\inf_\mu\int_{\T\times T\T} L-\eta d\mu$$
as well as the Mather set $\tilde{\mathscr M}(c)$ to be the set of support of measures attaining the above infimum. 

The main theorems are as follows. 
\begin{The}\label{ThmMain}Let $\mathcal H$ be the set of Tonelli Lagrangians satisfying that the set 
\begin{align*}
	\Lambda_{[c_1,c_2]}:=\bigcup_{c\in [c_1,c_2]}\tilde{\mathscr M}(c)
\end{align*}
is uniformly hyperbolic. Then the graph of $D\alpha:\ [c_1,c_2]\to \R$ is purely singularly continuous.   
   \end{The}


\begin{The}\label{thm:hausdorff-dim}
    Under  the assumption of the last theorem, the set of $c$ such that $D^2\alpha$ does not exist has zero Hausdorff dimension. 
\end{The}

Moreover, in the settings of KAM theorem, the next result shows that the devil's staircase cannot be complete. By a KAM circle, we mean a homologically nontrivial circle $\gamma$ that is invariant under the time-1 map of the Euler-Lagrange flow and the dynamics on  $\gamma$ is conjugate to a Diophantine rotation on circle by a $C^2$ diffeomorphism (see Section  \ref{sec: kam}). 
\begin{The}\label{ThmKAM}
    Suppose there is a positive Lebesgue measure set $P$ of Diophantine rotation numbers, such that each Mather set $\tilde{\mathscr M}_h$ (see \eqref{eq:mather-homology}), $h\in P$, is a KAM circle. Then $D\alpha$ has a nontrivial absolutely continuous part. 
\end{The}

Geometrically, this is also important. In metric geometry, there is an important quantity called \textit{stable norm}. Given $(M,g)$ is a closed Riemannian manifold, the stable norm $\sigma_g$ is defined as follows. For $h\in H_1(M,\Z)$,
\begin{align*}
    \sigma_g(h):=\inf\{l_g(\gamma):\gamma:S^1\to M, [\gamma]=h\in H_1(M,\Z)\}, 
\end{align*}
where $l_g(\gamma)$ is the length of the curve $\gamma$. For $h\in H_1(M,\Q)$, we can also write
\begin{align*}
	\sigma_g(h):=\inf\{l_g(\gamma)/n:\gamma:S^1\to M,n\in\N,[\gamma]=nh\in H_1(M,\Z)\}.
\end{align*}
Due to the convexity and positively 1-homogenity of $\sigma_g$ on $H_1(M,\Q)$, we have the unique continuous extension on $H_1(M,\R)$:
\begin{align*}
	\sigma_g(h):=\lim_{r\to h,r\in H_1(M,\Q)}\sigma_g(r).
\end{align*}
where $l_g(\gamma)$ is the length of the curve $\gamma$.

When treating one half the Riemannian metric as the Lagrangian, its Euler-Langrange flow is the same as the geodesic flow. We can also introduce $\beta$-function for Riemannian metrics. The relation of the stable norm $\sigma$ to the $\beta$ is $\beta=\sigma_g^2.$ With this identification, we also introduce the \textit{dual stable norm} $$\sigma_g^*(c):= \max_{\sigma_g(h)=1}\langle c, h\rangle.$$ 

The stable norm $\sigma_g$ has positively $1$-homogeneity, i.e. $\sigma_g(\lambda h)=|\lambda|\sigma_g(h)$. Thus, it is only interesting to study a level set of $\sigma_g$ and dually a level set of $\sigma_g^*$.  
\begin{The}\label{ThmStableNorm}
Let $\mathcal G$ be the set of Riemannian metrics on $\T^2$ such that for each $g\in \mathcal G,$ the Mather set $\tilde{\mathscr M}_h$ for all $h\in \mathbb{RP}^1$ is uniformly hyperbolic, then for each $g\in \mathcal G$,  the level set $(\sigma_g^*)^{-1}(1/2)$ as a radial graph over the unit circle is purely singularly continuous. 
\end{The}

There is a generalization of the theory of twist maps to the setting of minimal hypersurfaces in torus by Moser (\cite{MR847308}) and Bangert (\cite{MR991874}). There is an analogue of Theorem \ref{ThmDevil} in \cite{MR2197072}. We expect that there should be also a generalization of our main theorem to this setting.


The paper is organized as follows. In Section \ref{sec:pre}, we provide some preliminary of the Aubry-Mather theory on twist maps and the definitions of Mather set, Aubry set, etc. In Section \ref{SComplete}, we give the proof of the main theorems \ref{ThmMain}, \ref{thm:hausdorff-dim}, \ref{ThmStableNorm} on the completeness of the devil's staircase. In Section \ref{SIncomplete}, we prove Theorem \ref{ThmKAM} on the incompleteness of devil's staircase in the presence of KAM circles. Finally, we have two appendices containing materials on convex analysis and definition of twist maps. 

\medskip
\noindent\textbf{Acknowledgments.}
Jinxin Xue is supported by the National Natural Science Foundation of China  (Grant No.~12271285, 11790273), and by the New Cornerstone Investigator program and Xiaomi Foundation. Tianqi Shi is partially supported by the National Natural Science Foundation of China (Grant No.~12601348).  All authors would like to thank Professor Konstantin Khanin for discussions on related topics in circle maps, twist maps and spectral theory.

\section{Preliminaries}\label{sec:pre}
In this section, we provide some preliminaries on the Aubry-Mather theory of twist maps and Mather's generalization to general Tonelli Lagrangian systems. 
\subsection{The Aubry-Mather theory on twist maps}
We do not use the definition of a twist map directly so we put it in Appendix \ref{sec:twist map}. Instead, we invoke the following theorem relating a twist map to a Tonelli Hamiltonian system.
\begin{The}[Moser, \cite{MR863203}]\label{thm:moser}
	Given a twist map $\Gamma$ on $T^*\T=\T\times\R$, there is a Tonelli Hamiltonian $H(t,x,p):\T\times T^*\T\to\R$ whose time 1-map of the Hamiltonian flow is $\Gamma$. 
\end{The}

Let us state the main result of the Aubry-Mather theory for twist maps.  First, it was proved that each minimal configuration has a well-defined rotation number $\varrho=\lim\frac{x_n}{n}$(see \cite{MR719634,mackay1985lectures,MR1323222,MR945963}). The converse is also true: every real number is the rotation number of a minimal configuration. 
This means that each minimal configuration of atoms has a well-defined mean distance, that can be rational or irrational. It becomes particularly interesting when the mean distance is irrational, in which case, the locations of the atoms become incommensurate with the periodic potential. 
 
    Let $\varrho\in\R$ and $\mathscr Q$ be the set of minimal configurations, we denote 
	\begin{enumerate}[\rm (1)]
		\item by $\mathscr Q_\varrho:=\{\bm{x}\in \mathscr Q:\lim_n\frac{x_n}{n}=\varrho\}$ the set of minimal configurations with rotation number $\varrho$,
		\item and by $\mathbf Q_\varrho:=\{ \{\ldots,(y_i,x_i),\ldots\}:\bm {x}\in \mathscr Q_\varrho\}$ the union of the associated orbits of the twist map for each configuration in $\mathscr Q_\varrho$.
	\end{enumerate}
 
    \begin{The}[Classification of minimal configurations]\label{ThmAubryMather}\
	\begin{enumerate}[\rm (1)]
		\item $($Irrational rotation number$)$ If the rotation number $\varrho$ is irrational, then  
        \begin{enumerate}[\rm (a)]
            \item  either $\mathbf Q_\varrho$ is an invariant circle,
            \item or  $\mathbf Q_\varrho$ is a Cantorus, with possibly a set of nonrecurrent orbits. 
        \end{enumerate}
       
		\item $($Rational rotation number$)$ If the rotation number $\varrho$ is rational, then 
  \begin{enumerate}[\rm (a)]
      \item either $\mathbf Q_\varrho$ is an invariant circle consisting of periodic orbits,
      \item or  $\mathscr Q_\varrho=\mathscr Q_{\varrho^0}\cup \mathscr Q_{\varrho^+}\cup \mathscr Q_{\varrho^-}$ disjoint union: 
      \begin{enumerate}[\rm (i)]
          \item The set $\mathbf Q_{\varrho^0}$ consists of periodic orbits and there are gaps of $\mathscr Q_{\varrho^0}$ on $\Z\times \R$. 
          \item The set $\mathscr Q_{\varrho^+}$ $($respectively $\mathscr Q_{\varrho^-})$ consists of configurations lying in the gaps approaching the upper $($respectively lower$)$ configuration $\mathscr Q_{\varrho^0}$ in the future, and the lower $($respectively upper$)$ configuration $\mathscr Q_{\varrho^0}$ in the past.
      \end{enumerate}
  \end{enumerate}
	\item $($Total order property$)$
    \begin{enumerate}[\rm (a)]
			\item In case $(1)$ and $(2.a)$, the set $\mathscr Q_\varrho$ is totally ordered on $\Z\times \R$. 
			\item In case $(2.b)$, each of the sets $\mathscr Q_{\varrho^0}\cup \mathscr Q_{\varrho^+}$ and $\mathscr Q_{\varrho^0}\cup \mathscr Q_{\varrho^-}$ is totally ordered. 
		\end{enumerate}
\end{enumerate}
\end{The}
\subsection{The variational theory of Hamiltonian dynamics}\label{subsec:cvham}
 For a closed smooth Riemannian manifold $M$ and any $x\in M$, $T_xM$ and $T_x^*M$ are denoted by the tangent and cotangent spaces of $M$ at $x$, respectively. $TM$ and $T^*M$ are the tangent and cotangent bundle of $M$, respectively. Moreover, we consider the $m$-dimensional torus $\T^m\cong\R^m/\Z^m$ (especially $m=1,2$) as a prototypical example in the following. 

A $C^2$ function $H:\T\times T^*M\to\R$ is said to be a \textit{Tonelli Hamiltonian}, if it satisfies the following assumptions:
\begin{enumerate} [(H1)]
		\item the Hessian $D_{pp}^2H(t,x,p)$ is positively definite for all $(t,x,p)\in \T\times T^*M$;
		\item $\theta_0(|p|_x)-c_0\leqslant H(t,x,p)\leqslant \theta_1(|p|_x)+c_1$ for all $(t,x,p)\in \T\times T^*M$, where $\theta_0,\theta_1:[0,+\infty)\to[0,+\infty)$ are two superlinear functions and $c_0,c_1>0$.
	\end{enumerate}
A \textit{Lagrangian} $L:\T\times TM\to\R$ associated to Tonelli Hamiltonian $H$ is determined through the Legendre transformation (convex conjugate)
\begin{align*}
	L(t,x,v):=\max_{p\in T_x^*M}\left\{\langle p,v\rangle-H(t,x,p)\right\}.
\end{align*}
It is not hard to check that, $L$ is of class $C^2$ and also Tonelli:
\begin{enumerate} [(L1)]
	\item the Hessian $D_{vv}^2L(t,x,v)$ is positively definite for all $(t,x,v)\in \T\times TM$;
	\item $\theta_1^*(|v|_x)-c_1\leqslant L(t,x,v)\leqslant \theta_0^*(|v|_x)+c_0$ for all $(t,x,v)\in \T\times TM$, where $\theta_0^*,\theta_1^*:[0,+\infty)\to[0,+\infty)$ are the convex conjugate of $\theta_0$ and $\theta_1$, respectively.
\end{enumerate}

Given an absolutely continuous curve $\gamma:[a,b]\to M$, the action of $\gamma$ is defined by
\begin{align*}
	S(\gamma):=\int_a^bL(t,\gamma(t),\dot\gamma(t))\,dt.
\end{align*} 

Hereafter, we shall use $\Phi_L$ (resp. $\Phi_H$) to denote the Lagrangian (Hamiltonian) flow induced by Euler-Lagrange (Hamilton) equation. In this case, $\Phi_L^t(s,x,v)=(s+t,d\gamma(t;s,x,v))$, where $\gamma$ is the solution to Euler-Lagrangian equation with initial value $d\gamma(s)=(x,v)$. This is analogous to Hamiltonian flow $\Phi_H$. Sometimes, without causing confusion, we tend to view $\Phi_L$ ($\Phi_H$) as a flow acting on the (co)tangent bundle and thus we also denote $\Phi_L^{s,t}$ as the non-autonomous Lagrangian flow from time $s$ to $t$: $\Phi_L^{s,t}(x,v)=d\gamma(t;s,x,v)$. 


\subsubsection{(Co)homology approach: $\alpha$ and $\beta$-function}We denote $H^1(M,\R)$ as the first de Rham cohomology cohomology group of $M$. As for the case of $M=\T^m$, we have $H^1(\T^m,\R)\cong \R^m$. 
The $\alpha$\textit{-function} is defined as 
\begin{align}\label{eq:alpha}
	\alpha(c):=-\min_{\mu\in\mathfrak M_L}\int_{\T\times TM}L(t,x,v)-\eta_c(v)\,d\mu,\,\,\,\,\forall [\eta_c]=c\in H^1(M,\R),
\end{align}where $\mathfrak M_L$ is all Borel probability measures on $\T\times TM$ invariant by the Euler-Lagrange flow $\Phi_L$. 

In the case of Tonelli Lagrangian $L$, one can obtain that $\alpha$-function is convex, finite everywhere with superlinear growth. Therefore, the Legendre transformation generates a function $\beta: H_1(M,\R)\to\R$
\begin{align*}
	\beta(h):=\max_{c\in H^1(M,\R)}\left\{\langle h,c\rangle-\alpha(c)\right\},
\end{align*}
which is also finite everywhere and convex with superlinear growth, called $\beta$\textit{-function}. There is a direct and intrinsic way to define the $\beta$-function. Here $\langle\cdot,\cdot\rangle$ denotes the canonical pairing between (de Rham) cohomology and homology. 
\begin{Pro}[\cite{Mather1991}]Assume that $\mu\in\mathfrak M_L$ and $[\eta]=c\in H^1(M,\R)$.
\begin{enumerate}[\rm (1)]
	\item If $\eta$ is exact, i.e., $[\eta]=0$, then $\int_{TM}\eta\,d\mu=0$;
	\item There exists $\rho(\mu)\in H^1(M,\R)$ (called the rotation vector of $\mu$) such that
         \begin{align*}
	        \langle[\eta],\rho(\mu)\rangle=\int_{TM}\eta\,d\mu
         \end{align*}for every $\eta$ is a closed 1-form on $M$;
    \item For $h\in H^1(M,\R)$, we have some $\mu\in\mathfrak M_L$ with $\rho(\mu)$. Furthermore, 
    \begin{align}\label{eq:beta}
    	\beta(h)=\min\left\{\int_{\T\times TM}L\,d\mu:\rho(\mu)=h\right\}.
    \end{align}  
\end{enumerate}	
\end{Pro}
A specific class of $\mu\in\mathfrak M_L$ is an ergodic measure. Suppose $d\gamma$ is a trajectory of flow $\Phi_L$, i.e.,  $(t,d\gamma(t))=\Phi_L^t(0,d\gamma(0))$, $t\in\R$.  For any $T>0$, we define a Borel probability measure $\mu_T$ on $TM$ by Riesz representation theorem
\begin{align*}
	\mu_T(f):=\frac{1}{2T}\int_{-T}^Tf(\Phi_L^t(0,d\gamma(0)))\,dt,\,\,\,\,\forall f\in C(\T\times TM).
\end{align*}Since $d\gamma$ is supported in a compact subset of $TM$, As $T\to\infty$, we can extract a sequence of $\{\mu_T\}_{T>0}$ such that it weakly converges to some Borel probability measure $\mu\in\mathfrak M_L$. Furthermore, we can also get 
\begin{Pro}[\protect{\cite[Proposition 1 and 2]{Mather1991}}]\label{pro:mather1991}Suppose $\widetilde M$ is the abelian covering space of $M$.
 Let $\tilde\gamma_i:[a_i,b_i]$ be a sequence of global minimizers such that $\frac{\tilde\gamma(b_i)-\tilde\gamma(a_i)}{b_i-a_i}\to h\in H_1(M,\R)$ and $b_i-a_i\to\infty$ as $i\to\infty$, then $\frac{1}{b_i-a_i}S(\tilde\gamma_i)\to\beta(h)$.
\end{Pro}
According to Proposition \ref{pro:mather1991},
\begin{align}\label{eq:beta-2}
	\beta(h)=\min\left\{\liminf_{T\to\infty}\frac{1}{2T}\int_{-T}^TL(\Phi_L^t(0,d\tilde\gamma(0)))\,dt: \rho(\gamma)=h\right\},
\end{align}
here $\tilde\gamma$ is the lifting of $\gamma$ on $\widetilde M$ and $\rho(\gamma):=\lim_{T\to\infty}\frac{\tilde\gamma(T)-\tilde\gamma(-T)}{2T}$.

The next result includes zeroth and first order properties of the $\alpha$ and $\beta$ functions in the case of $M=\T$, which can be applied to twist map by Theorem \ref{thm:moser}. 
\begin{The}[\cite{MR719055,MR1384394,MR967638}]\label{thm:beta-striconvex}Suppose $M=\T$. 
\begin{enumerate}[\rm (1)]
	\item $\beta$ is always differentiable on $\R\setminus\Q$ and generically is not differentiable on $\Q$;
	\item $\beta$ is strictly convex everywhere. Therefore, $\alpha\in C^1(\R)$.
\end{enumerate}

\end{The}

\subsubsection{Mather set and Aubry set}
\begin{defn}[Mather measure and Mather set]Suppose $c\in H^1(M,\R)$. The \textit{Mather measure} (for cohomology $c$) is the probability measure which is the minimizer of \eqref{eq:alpha}. As the subset of $\T\times TM$, the \textit{Mather set} for $c\in H^1(M,\R)$ is defined as the union of the support of all Mather measures $\mu_c$ for cohomology $c$
	\begin{align*}
		\tilde{\mathscr M}(c):=\bigcup\supp(\mu_c).
	\end{align*}
Moreover, we can also introduce the Mather set $\tilde{\mathscr M}_h$ associated to each rotation vector $h\in H_1(M,\R)$ as the union of the support of the measures attaining \eqref{eq:beta}. 

We denote $\mathscr M(c)$ and $\mathscr M_h$ as the projection of $\tilde{\mathscr M}(c)$ and $\tilde{\mathscr M}_h$ on $M$ and call them \textit{projected Mather sets} of cohomology $c$ and homology $h$ respectively.
\end{defn} The relation between the two notions of Mather sets is
\begin{align}\label{eq:mather-homology}
	\tilde{\mathscr M}_h=\bigcup_{c\in D^-\beta(h)}\tilde{\mathscr M}(c),\,\,\,\,\tilde{\mathscr M}(c)=\bigcup_{h\in D^-\alpha(c)}\tilde{\mathscr M}_h.
\end{align} 

Assuming that $\gamma:[t_0,t_1]\to M$ is absolutely continuous and $c\in H^1(M,\R)$ is a cohomology class, we denote \textit{the action of $\gamma$ associated to cohomology $c$} as
\begin{align}\label{eq:action with c}
	S_c(\gamma):=\int_{t_0}^{t_1}(L(t,d\gamma(t))-\eta_c(d\gamma(t))+\alpha(c))\,dt,
\end{align}where $\eta_c$ is a closed 1-form and $[\eta_c]=c$. 


\begin{defn}[Ma\~n\'e potential, semi-static curve and Ma\~n\'e set]Suppose $c\in H^1(M,\R)$. 
	\begin{enumerate}[\rm (1)]
		\item Let $(t_0,x_0),(t_1,x_1)\in\T\times M$. We denote by
		\begin{align*}
			\Psi_c((t_0,x_0),(t_1,x_1)):=\inf_{\tau_0-t_0,\tau_1-t_1\in\Z}\inf_{\substack{\gamma(\tau_0)=x_0,\gamma(\tau_1)=x_1}}S_c(\gamma)
		\end{align*}the free-time minimization action, or \textit{Ma\~n\'e potential}. 
	   \item A curve $\gamma:\R\to M$ is called $c$\textit{-semi static}, if 
	   \begin{align*}
	   	S_c(\gamma|_{[t_0,t_1]})=\Psi_c((t_0,\gamma(t_0)),(t_1,\gamma(t_1))),\,\,\,\,\forall t_0,t_1\in\R.
	   \end{align*}The \textit{Ma\~n\'e set} for cohomology class $c$ is defined as
	   \begin{align*}
	   	\tilde{\mathscr N}(c):=\bigcup\left\{(t,d\gamma(t)):\gamma\text{ is }c\text{-semi static},t\in\R\right\}\subset\T\times TM.
	   \end{align*}
	\end{enumerate}
\end{defn}
\begin{defn}[Static curve and Aubry set]\label{defn:Aubry}Suppose $c\in H^1(M,\R)$. A $c$-semi static curve is called \textit{$c$-static}, if 
\begin{align*}
	S_c(\gamma|_{[t_0,t_1]})=-\Psi_c((t_1,\gamma(t_1)),(t_0,\gamma(t_0)))
\end{align*}
	additionally holds true for any $t_0,t_1\in\R$. The \textit{Aubry set} for the cohomology class $c$ is defined as 
	\begin{align*}
		\tilde{\mathscr A}(c):=\left\{(t,d\gamma(t)):\gamma\text{ is }c\text{-static},t\in\R\right\}\subset\T\times TM.
	\end{align*}
\end{defn}
Likewise, $\mathscr N(c)=\pi(\tilde{\mathscr N}(c))$ (resp. $\mathscr A(c)=\pi(\tilde{\mathscr A}(c))$) is the projection of $\tilde{\mathscr N}(c)$ (resp. $\tilde{\mathscr A}(c)$) on $M$, called \textit{projected Ma\~n\'e }(resp. \textit{Aubry}) \textit{set}.
Some properties are known and we recall them below.
\begin{Pro}Suppose $c\in H^1(M,\R)$.
	\begin{enumerate}[\rm (1)]
        \item A $c$-static loop $\gamma$ satisfies $S_c(\gamma)=0$;
		\item $\tilde{\mathscr M}(c)\subset\tilde{\mathscr A}(c)\subset\tilde{\mathscr N}(c)$;
		\item $c\mapsto\tilde{\mathscr N}(c)$ is upper semi-continuous in the sense of Hausdorff topology.
	\end{enumerate}
\end{Pro}
Whenever no confusion arises, we also use \(\tilde{\mathscr M}(c)\) to denote its time-zero section, regarded as an invariant subset of \(T\mathbb T\) under the time-one map \(\Phi_L^{0,1}\).
\begin{Pro}[Lipschitz graph]\label{pro:mather-lip}
	There exists $K_0>1$ determined by Lagrangian $L$, such that for any $c\in[c_1,c_2]$, 
	\begin{align}\label{eq:equi-lip-graph}
		d_{T\T}((x,v),(y,w))\leqslant K_0\cdot d_{\T}(x,y),
	\end{align}
	where $(x,v),(y,w)\in \tilde{\mathscr M}(c)$. Thus, the projection  $\pi|_{\tilde{\mathscr M}(c)}: \tilde{\mathscr M}(c)\to {\mathscr M}(c)$ is an injection. 
\end{Pro}
\begin{proof}
	For any fixed $c\in[c_1,c_2]$, the Lipschitz graph theorem (see \cite{Mather1991}) implies that $\tilde{\mathscr M}(c)$ is a Lipschitz graph of its projection $\mathscr M(c)$. Since $[c_1,c_2]$ is a compact interval, and $L$ is Tonelli, the a priori compactness is given as
	$\tilde{\mathscr M}(c)\subset\mathcal K\subset\subset T\T$.
	The Lipschitz constant of the graph for any fixed $c\in[c_1,c_2]$ is dependent on the Lipschitz property of the differential of the backward weak KAM solution $u_c$ for the time-one Lax–Oleinik operator associated with $L-cv$. We can obtain from \cite{Bernard2008_1} that $u_c$ are both semiconcave and semiconvex uniformly to $c\in[c_1,c_2]$ on Lipschitz graph, which is equivalent to the fact that $Du_c$ is uniformly Lipschitz with respect to $c$. Consequently, we can find a uniform Lipschitz constant $K_0>0$ satisfying \eqref{eq:equi-lip-graph} for any $c\in[c_1,c_2]$.  
\end{proof}
\section{Completeness of devil's staircase in the hyperbolic case}\label{SComplete}
In this section, we consider the case of $M=\T$ equipped with flat metric and $L=L(t,x,v):\T\times T\T\to\R$ is a time 1-periodic nonautonomous  Tonelli Lagrangian. The Abelian covering space of $M$ is $\widetilde M=\R$. Recall that the (projected) Mather set of rotation number $h\in H_1(\T,\R)\cong\R$ is denoted by $\mathscr M_h$.

Without loss of generality, we always assume a rational number $p/q\in\Q$ satisfies $q>0$ and $\gcd(p,q)=1$.
For each rotation number $p/q\in\Q$, the set of all periodic minimal curves is not empty. In other words, 
\begin{align*}
	\mathscr M_{p/q}:=\left\{\gamma\in\mathscr A_{p/q}:\tilde\gamma(t+q)=\tilde\gamma(t)+p\right\}\neq\varnothing,
\end{align*}where $\tilde\gamma$ is the lift of $\gamma$ from $\T$ to $\R$. 
\subsection{The proof of main theorems on complete devil's staircase}
Before proceeding, we give a remark on Theorem \ref{ThmMain} to understand the proof of it.
\begin{Rem}\label{rem:mainthm-explain}
	Due to the strict convexity of $\beta$-function, $\alpha\in C^1(H^1(\T,\R))$ and then $D\alpha$ is continuous and monotone non-decreasing. As a result, for $[c_1,c_2]\subset H^1(\T,\R)$, $D\alpha([c_1,c_2])=[D\alpha(c_1),D\alpha(c_2)]$. Therefore, we focus on the case of rotation number $h\in[0,1]$ in the subsequent discussion without loss of generality. Moreover, assumptions in Theorem \ref{ThmMain} is equivalent to the uniform hyperbolicity of  $\tilde{\mathscr M}_h$ for each $h\in[0,1]$. For this purpose, we suppose $(D\alpha)^{-1}([0,1])=[c_1,c_2]$ for convenience. 
\end{Rem}
The main Theorem \ref{ThmMain} can be obtained basically from the following exponential flatness results of the $\beta$-function. Estimates of such type first appeared in Aubry's work \cite{10.1007/BFb0093512,MR1219348} in the setting of Frenkel-Kontorova model assuming strong hyperbolicity arised from the anti-integrability limit. The anti-integrability limit allows Aubry to use symbolic coding of orbits in the Cantorus, which is not available in general. 
It also appeared in \cite{MR3622271} in the setting of stable norms of Finsler metrics on $\T^2$, where some special properties of the Finsler metrics, such as the 1-homogeneity, are essentially used. In this section, we work in the more general setting of Tonelli Lagrangians. With the language of  modern Mather theory, we get a significantly concise proof. 

\begin{The}[Exponential flatness of the $\beta$-function]\label{thm:exp-flat-all}
	Suppose $\mathscr M_{p/q}$ is uniformly hyperbolic for Euler-Lagrange flow $\Phi_L$ with respect to all $p/q\in[0,1]\cap\Q$.
	Then there exist $C,\lambda>0$ depending on $L$, as well as $\varepsilon>0$ depending on $q$, such that for any $\delta\in (0,\varepsilon)$, we have
\begin{align}\label{eq:exp-flat-all}
	\beta(p/q+\delta)-\beta(p/q)-\beta_+'(p/q)\delta\leqslant Cq\delta e^{-\frac{\lambda}{4N_{[c_1,c_2]}q\delta}}.
\end{align}
 Here $\beta_+'$ denotes the right derivative of $\beta$ (see Appendix \ref{sec:convex}). 
More precisely, the constants mentioned above can be chosen as 
\begin{align*}
	C:=4(\Lip(L,K)(1+1/\lambda)+1)C_0N_{[c_1,c_2]},\,\,\,\varepsilon:=1/(4qN_{[c_1,c_2]}),
\end{align*}where $C_0$ and $\lambda$ are the uniformly hyperbolicity constants of Euler-Lagrange flow $\Phi_L$ for all $p/q\in[0,1]$ (see Proposition \ref{pro:structure}); $\Lip(L,K)$ is the Lipschitz constant of $L$ on some compact subset $K\subset \T\times T\T$ determined by bounded homology interval $[0,1]$; $N_{[c_1,c_2]}$ is a uniform upper bound of $\#(\mathscr M_{p/q})$ for $p/q\in\Q\cap D\alpha([c_1,c_2])$ (see Proposition \ref{pro:uni-num-bounded}).
\end{The}


The next result is due to Aubry, which gives a criterion to establish the completeness of the devil's staircase from the flatness estimate. 
\begin{Lem}[\cite{10.1007/BFb0093512,MR1219348}]\label{lem:var=0}
	Assume that $\phi$ is a strictly convex function defined on $\R$. If for all reduced fraction $p/q\in\Q$, we have
	\begin{align}\label{eq:u-flat}
		0<\phi(x)-\phi(p/q)-\phi_{+}'(p/q)(x-p/q)<u_q(x-p/q),\,\,\,\,\forall x\in(p/q,p/q+1/(4qN_{[c_1,c_2]})),
	\end{align}
 where $u_q$ satisfies the condition of convergence
	\begin{align}\label{eq:u-sum-finite}
		\sum_{q=1}^\infty q^{2+\nu}u_q\left(\frac{1}{q^{1+\nu}}\right)<\infty
	\end{align}
	for some $\nu\in (0,1)$, then 
	\begin{enumerate}[\rm (1)]
		\item $\Var(\phi_{+}',\R\setminus\Q)=0$;
		\item suppose $\psi$ is the convex conjugate of $\phi$, then $\psi$ has a singular continuous derivative function, i.e., $D\psi$ is (non-constant and) continuous and $D^2\psi=0$ almost everywhere in the sense of Lebesgue measure.
	\end{enumerate}
    \end{Lem}
    We further improve the last result to the following. 
\begin{Lem}\label{lem:haus-dim=0}Under the conditions of Lemma \ref{lem:var=0}, 
    furthermore, if $u_q$ in \eqref{eq:u-flat} additionally satisfies $q^{1+\nu}u_q\left(\frac{2}{q^{1+\nu}}\right)\to 0$  as $q\to\infty$ and 
	\begin{align}
		\label{eq:u-sum-finite-theta}
		\sum_{q=1}^\infty q^{(1+\nu)\theta+1}u_q^\theta\left(\frac{1}{q^{1+\nu}}\right)<\infty,\,\,\,\,\forall\theta\in(0,1],
	\end{align}
	then the set of $y$ such that $D^2\psi$ does not exist has zero Hausdorff dimension. 
\end{Lem}
\begin{proof}[Proof of Theorem \ref{ThmMain} and \ref{thm:hausdorff-dim}]
    According to the assumptions in Theorem \ref{ThmMain} and \ref{thm:exp-flat-all}, $\beta$ satisfies \eqref{eq:exp-flat-all}. If we set $u_q(x):=Cqxe^{-\lambda/(4qx)}$, then $\beta$ holds \eqref{eq:u-flat}. Consequently, since $\beta$ is a strictly convex function by Theorem \ref{thm:beta-striconvex}, Lemma \ref{lem:var=0} claims that $\alpha$ as the convex conjugate of $\beta$ has a singular continuous derivative. Namely, $D\alpha$ is a complete devil's staircase. 
    Also, $u_q$ is applied Lemma \ref{lem:haus-dim=0}, which concludes Theorem \ref{thm:hausdorff-dim}.
\end{proof}

\begin{proof}[Proof of Theorem \ref{ThmStableNorm}]
    Theorem \ref{ThmStableNorm} follows from similar argument as above, provided we replace the exponential flatness Theorem \ref{thm:exp-flat-all} by the following result proved in \cite{MR3622271}.
\end{proof}
\begin{Lem}[\cite{MR3622271}]
    Let $h\in\R^2\setminus\{0\}$ with rational or infinite slope. Assume that the set of periodic minimal geodesics $\mathscr M_h$ is uniformly hyperbolic for the geodesic flow $\Phi_F$ with a Finsler metric $F$ on $\T^2$. Then there exist constants $C,\lambda,\varepsilon>0$, such that for all $v\in\R^2$ with Euclidean norm $|v|<\varepsilon$,
    \begin{align*}
        \sigma_F(h+v)-\sigma_F(h)-[\partial^+\sigma_F(h)](v)\leqslant C|v|e^{-\frac{\lambda}{|v|}}.
    \end{align*}
where $\partial^+\sigma_F$ is the forward directional derivative of $\sigma_F$ defined as
    \begin{align*}
    	[\partial^+\sigma_F(h)](v):=\lim_{t\to 0^+}\frac{\sigma_F(h+tv)-\sigma_F(h)}{t}.
    \end{align*}
\end{Lem}
The remainder of this section is devoted to the proof of Theorem \ref{thm:exp-flat-all}, Lemma \ref{lem:var=0} and \ref{lem:haus-dim=0}. 
\subsection{The exponential flatness of $\beta$-functions} In this subsection, our aim is to give the proof of Theorem \ref{thm:exp-flat-all}. The next proposition shows the structure of the Aubry set and the Mather set, which is an analog to Theorem \ref{ThmAubryMather}.
\subsubsection{The structure of $\mathscr M_{p/q}$} Before studying the properties of the $\beta$-functions, we give some quantitive analysis on the amount of orbits and the structure in $\mathscr M_{p/q}$ under hyperbolicity assumptions in Theorem \ref{ThmMain}. 
\begin{Pro}\label{pro:uni-num-bounded}
 Under assumptions in Theorem \ref{ThmMain}, there exists $N_{[c_1,c_2]}\in\N^*$ such that for any $p/q\in\Q\cap D\alpha([c_1,c_2])$, $\#(\mathscr M_{p/q})<N_{[c_1,c_2]}$.  
\end{Pro}

\begin{Lem}\label{lem:stable-mani}
	There exists $\varepsilon_*>0$, such that: if for any $(x,v),(y,w)\in\Lambda_[c_1,c_2]$. 
	\begin{align}\label{eq:expansive}
		d_{T\T}\left((\Phi_{L}^{0,1})^n(x,v),(\Phi_{L}^{0,1})^n(y,w)\right)\leqslant \varepsilon_*,\,\,\,\,\,\,\forall n\in\Z,
	\end{align}
	then $(x,v)=(y,w)$. 
\end{Lem}
%
%
%
%
%
%
%
\begin{proof}
	Since \(\Lambda_{[c_1,c_2]}\) is a compact uniformly hyperbolic invariant set for the \(C^1\) diffeomorphism \(\Phi_{L}^{0,1}\), the uniform Hadamard–Perron theorem (\cite{Katok_Hasselblatt_book}), applied to the bi-infinite sequence of local coordinate representations of \(\Phi_{L}^{0,1}\) along the orbit $\{[\Phi_{L}^{0,1}]^n(x,v)\}_{n\in\Z}$ and, for the unstable manifolds, to the corresponding inverse sequence, gives constants $\rho>0,$ and $\delta>0$, independent of \((x,v)\in\Lambda_{[c_1,c_2]}\), together with local stable and unstable disks $W_\rho^s(x,v)$ and $W_\rho^u(x,v)$. 
More precisely, the dynamical characterization in the Hadamard–Perron theorem gives
\begin{align*}
	d_{T\mathbb T}\left([\Phi_{L}^{0,1}]^n(y,w),[\Phi_{L}^{0,1}]^n(x,v)\right)<\delta,\,\,\,\,\,\,\forall n\geqslant 0
\end{align*}
only if $(y,w)\in W_\rho^s(x,v)$, and similarly,
\begin{align*}
	d_{T\mathbb T}\left([\Phi_{L}^{0,1}]^{n}(y,w),[\Phi_{L}^{0,1}]^{n}(x,v)\right)<\delta \,\,\,\,\,\,\forall n\leqslant 0
\end{align*}
only if $(y,w)\in W_\rho^u(x,v)$. The uniformity of \(\rho\) and \(\delta\) with respect to \(x\) follows from the compactness of \(\Lambda_{[c_1,c_2]}\) and the uniform hyperbolicity of the splitting $T_{\Lambda_{[c_1,c_2]}}(T\mathbb T)=E^u\oplus E^s$.

After decreasing \(\rho\) uniformly if necessary, these local disks have the unique-intersection property
\begin{align*}
	W_\rho^s(x,v)\cap W_\rho^u(x,v)=\{(x,v)\},\,\,\,\,\,\,\forall (x,v)\in\Lambda_{[c_1,c_2]}.
\end{align*}
Indeed, in local coordinates adapted to $T_{(x,v)}(T\mathbb T)=E_{(x,v)}^u\oplus E_{(x,v)}^s$, the local disks supplied by the Hadamard–Perron theorem can be written as graphs
\begin{align*}
	W_\rho^s(x,v) = \left\{\bigl(h_{(x,v)}^s(\mathbf v),\mathbf v\bigr): \mathbf v\in E_{(x,v)}^s,\ \|\mathbf v\|<\rho \right\},\quad W_\rho^u(x,v) = \left\{ \bigl(\mathbf u,h_{(x,v)}^u(\mathbf u)\bigr): \mathbf u\in E_{(x,v)}^u,\ \|\mathbf u\|<\rho \right\},
\end{align*}
where $h_{(x,v)}^s(0)=0$, $h_{(x,v)}^u(0)=0$ and after a uniform reduction of \(\rho\),
$\Lip(h_{(x,v)}^s),\Lip(h_{(x,v)}^u)\leqslant \kappa\in(0,1)$
independently of \((x,v)\). If a point with local coordinates \((u,v)\) belongs to both local disks, then $\mathbf u=h_{(x,v)}^s(\mathbf v)$ and $\mathbf v=h_{(x,v)}^u(\mathbf u)$.
Consequently, $\|\mathbf u\| \leqslant\kappa\|\mathbf v\| \leqslant\kappa^2\|\mathbf u\|$. Since \(\kappa<1\), this implies \(\mathbf u=\mathbf v=0\). Hence the only common point of the two local disks is \((x,v)\).

Now set $\varepsilon_*:=\frac{\delta}{2}$. Suppose that \((x,v),(y,w)\in\Lambda_{[c_1,c_2]}\) satisfy
\begin{align*}
	d_{T\mathbb T}\bigl([\Phi_{L}^{0,1}]^n(y,w),[\Phi_{L}^{0,1}]^n(x,v))<\varepsilon_*<\delta,\,\,\,\,\,\,\forall n\in\mathbb Z. 
\end{align*}
For every \(n\geq0\), we then have $(y,w)\in W_\rho^s(x,v)$ by the forward dynamical characterization of the local stable disk, Likewise, for every \(n\geqslant 0\), we can also that $(y,w)\in W_\rho^u(x,v)$. Therefore, $(y,w)\in W_\rho^s(x,v)\cap W_\rho^u(x,v)=\{(x,v)\}$ and hence $(x,v)=(y,w)$. 
\end{proof}
\begin{proof}[Proof of Proposition \ref{pro:uni-num-bounded}]
	We arbitrarily choose $N\in\N$ periodic minimal orbits in $\tilde{\mathscr M}_{p/q}\subset \tilde{\mathscr M}(c)$ for some $c\in[c_1,c_2]$ (by \eqref{eq:mather-homology}). We aim to claim that such amount $N$ must be less than some integer independent to the rational rotation numbers associated to cohomology in $[c_1,c_2]$. Since each orbit is $q$-periodic, there are exactly $q$ points on the time-zero section in the phase space. Then $Nq$ points form on the time-zero section by $N$  $q$-periodic orbits of $\tilde{\mathscr M}_{c}$. We denote set $Z\subset\tilde{\mathscr M}(c)$ the union of these $Nq$ points and define set $A:=\pi Z\subset\mathscr M(c)$ given by projection $\pi:T\T\to\T$. Thus, $\#Z=\#A=Nq$.
	
	Consider the map
	\begin{align*}
		F:=\pi\circ \Phi_{L}^{0,1}\circ(\pi|_{Z})^{-1}:A\to A.
	\end{align*}
	Then $F$ is a permutation of $A$, i.e., $F$ is a bijection. Furthermore, $F$ also keeps the cyclic order because of the Aubry-Le Daeron crossing lemma (\cite{MR719634}). If we index the points in $A:=\{a_0,a_1,\cdots,a_{Nq-1}\}$ with counterclockwise order in $\T$, then we assert that there exists $l\in\N$ such that $F(a_i)=a_{i+l}$ with $i,i+l\mod Nq$. Moreover, owing to the $q$-periodicity of orbits, for $\pi(z_i)=a_i$, $F^q(a_i)=\pi([\Phi_{L}^{0,1}]^q(z_i))=\pi(z_i)=a_i$. Conversely, since $\gcd(p,q)=1$, $q$ is the minimal period of orbits in $\mathscr M_{p/q}$, which means the cycle decomposition of $F$ consists $N$ disjoint $q$-cycles. 
	
	The points of $A$ separate $\T$ into $Nq$ intervals $I_i:=(a_i,a_{i+1})_{\T}$ with $i=0,1,\cdots,Nq-1$. Then $\sum_{i=0}^{Nq-1}|I_i|=1$. Because $F(a_i)=a_{i+l}$ and $F(a_{i+1})=a_{i+1+l}$, we define $\sigma(I_i):=I_{i+l}$. Accordingly, $\sigma^q(I_{i})=I_{i+ql}=I_i$ and $Nq$ intervals are partitioned into $N$ cycles of intervals, each of length $q$. 
	
	Choose $\eta:=\min\{1/2,\varepsilon_*/K_0\}$. Now we assume, for the sake of contradiction, that $q$ intervals in one interval cycle has length less than $\eta$. For example, $0<|I_i|<\eta$ means that for $a_i=\pi z_i$ and $a_{i+1}=\pi z_{i+1}$, $a_i\neq a_{i+1}$, as well as for all $n\in\Z$, $|\sigma^n(I_i)|<\eta$ by assumption. Owing to $\eta<1/2$, 
	\begin{align*}
		d_{\T}(\pi[\Phi_L^{0,1}]^n(z_i),\pi[\Phi_L^{0,1}]^n(z_{i+1}))<\eta,\,\,\,\,\,\,\forall n\in\Z.  
	\end{align*}
	This is because the length of an arc with a length less than \(1/2\) is exactly equal to the Riemann distance on the circle between the two endpoints. Recall that $[\Phi_L^{0,1}]^{n}(z_i),[\Phi_L^{0,1}]^{n}(z_{i+1})\in\tilde{\mathscr M}(c)$, then as a result of Proposition \ref{eq:equi-lip-graph}, for all $n\in\Z$, 
	\begin{align*}
		d_{T\T}([\Phi_L^{0,1}]^{n}(z_i),[\Phi_L^{0,1}]^{n}(z_{i+1}))\leqslant K_0\cdot d_{\T}(\pi[\Phi_L^{0,1}]^n(z_i),\pi[\Phi_L^{0,1}]^n(z_{i+1}))\leqslant K_0\cdot\eta<\varepsilon_*.
	\end{align*}
	Applying Lemma \ref{lem:stable-mani}, we know that $z_i=z_{i+1}$, which leads to a contradiction since $a_i\neq a_{i+1}$. Consequently, at least one interval in each interval cycle is longer than $\eta$. 
	
	Since we have $N$ interval cycles, we can find interval $J_k$ with length $|J_k|\geqslant \eta$ for $k=1,\cdots, N$. These $J_k$ are from different interval cycles then they are pairwise intersect. Meanwhile,
	\begin{align*}
		N\eta=\sum_{k=1}^N \eta\leqslant \sum_{k=1}^N |J_k|\leqslant 1.
	\end{align*}
	$N\leqslant \lfloor\eta^{-1}\rfloor$ which is the desired bound. Note that the definition of $\eta$ is only associated to the Lagrangian $L$ and its Euler-Lagrange flow $\Phi_L$ on a compact subset of $TM$ obtained from the compactness of $[c_1,c_2]$. This fact allows us to get a uniform upper bound $N_{[c_1,c_2]}:=\lfloor\eta^{-1}\rfloor+1$ of $\#(\mathscr M_{p/q})$ for $p/q\in\Q\cap D\alpha([c_1,c_2])$.
\end{proof}

\begin{Pro}\label{pro:structure}
	Suppose $\#(\mathscr M_{p/q})<\infty$. Let $\gamma_{p/q,j}\in\mathscr M_{p/q}$, $j=0,\cdots,\#(\mathscr M_{p/q})-1$, be the $q$-periodic minimal orbits on $\T$ with order $\gamma_{p/q,j}<\gamma_{p/q,j+1}$ and $\tilde\gamma_{p/q,j}^{k}$, $k\in\N^*$ be the lifts of $\gamma_{p/q,j}$ from $\T$ to $\R$ with order $\tilde\gamma_{p/q,j}^{k}<\tilde\gamma_{p/q,j+1}^{k+1}$. Under assumptions of Theorem \ref{thm:exp-flat-all}, there exists a heteroclinic $\tilde\xi_{p/q,j}^{k+1}$, which satisfies $\tilde\gamma_{p/q,j}^{k}<\tilde\xi_{p/q,j+1}^{k}<\tilde\gamma_{p/q,j+1}^{k}$ and converges forward to $\tilde\gamma_{p/q,j+1}^{k+1}$ and backward to $\tilde\gamma_{p/q,j}^{k}$ exponentially for each $k\in\N$.  
\end{Pro}
\begin{proof}
	Invoking the crossing lemma by Aubry-Le Daeron (\cite{MR719634}), we obtain that $\tilde\gamma_{p/q,j}^k$ as the lifts of $\gamma_{p/q,j}$, are totally ordered:
	\begin{align*}
		\cdots<\tilde\gamma_{p/q,0}^{k}<\tilde\gamma_{p/q,1}^{k}<\cdots<\tilde\gamma_{p/q,j-1}^{k}<\tilde\gamma_{p/q,j}^{k}<\cdots<\tilde\gamma_{p/q,\#(\mathscr M_{p/q})}^{k}<\tilde\gamma_{p/q,1}^{k+1}<\cdots
	\end{align*}
	and
	\begin{align*}
		\cdots<\gamma_{p/q,j}=\tilde\gamma_{p/q,j}^{1}<\cdots<\tilde\gamma_{p/q,j}^{k-1}<\tilde\gamma_{p/q,j}^{k}<\cdots<\tilde\gamma_{p/q,j
		}^{q}=\tilde\gamma_{p/q,j}^1+1<\cdots.
	\end{align*}
	Furthermore, $\tilde\gamma_{p/q,j}^k<\tilde\gamma_{p/q,j+1}^k$ and $\tilde\gamma_{p/q,\#(N_{p/q})}^k<\tilde\gamma_{p/q,}^{k+1}$ are pairs of $p/q$-neighboring minimal orbits: here are no further minimal orbits lifted from $\mathscr M_{p/q}$ between them. Due to the uniform hyperbolicity of $\mathscr M_{p/q}$ and the periodicity of $L$, given $t_{j,k}\in\Z$ for $j=0,\cdots,\#(\mathscr M_{p/q})-1$ and $k=0,\cdots,q-1$, there exists a heteroclinic orbit $\tilde\xi_{p/q,j+1}^{k}$ lifted from $\mathscr A_{p/q}\setminus\mathscr M_{p/q}$, which is $\alpha$-asymptotic to $\tilde\gamma_{p/q,j}^{k}$ and $\omega$-asymptotic to $\tilde\gamma_{p/q,j+1}^{k}$: we can find some $C_0,\lambda>0$ depending on the Euler-Lagrange flow $\Phi_L$, such that for $k=0,\cdots,q-1$, 
	\begin{align}\label{eq:exp-convergence-1}
	\left\{
	\begin{array}{ll}
		d_{T\T}\left(d\tilde\xi_{p/q,j}^{k}(t),d\tilde\gamma_{p/q,j}^{k}(t)\right)\leqslant C_0e^{-\lambda (t-t_{j,k})},& t\geqslant t_{j,k};\\
		d_{T\T}\left(d\tilde\xi_{p/q,j}^{k}(t),d\tilde\gamma_{p/q,j-1}^{k}(t)\right)\leqslant C_0e^{\lambda (t-t_{j,k})},& t\leqslant t_{j,k},\\
	\end{array}
	\right. (j=1,\cdots,\#(\mathscr M_{p/q})-1)
	\end{align}
and
\begin{align}\label{eq:exp-convergence-2}
\left\{
	\begin{array}{ll}
	d_{T\T}\left(d\tilde\xi_{p/q,\#(\mathscr M_{p/q})}^{k}(t),d\tilde\gamma_{p/q,0}^{k+1}(t)\right)\leqslant C_0e^{-\lambda (t-t_{\#(\mathscr M_{p/q}),k})},& t\geqslant t_{\#(\mathscr M_{p/q}),k};\\
		d_{T\T}\left(d\tilde\xi_{p/q,\#(\mathscr M_{p/q})}^{k}(t),d\tilde\gamma_{p/q,\#(\mathscr M_{p/q})-1}^{k}(t)\right)\leqslant C_0e^{\lambda (t-t_{\#(\mathscr M_{p/q}),k})},& t\leqslant t_{\#(\mathscr M_{p/q}),k},\\
\end{array}
\right.
\end{align}
where $d_{T\T}$ is the metric of $T\T$ induced by flat metric of $\T$ and Euclid metric of $\R$.
\end{proof}
	Theorem \ref{thm:exp-flat-all} can be proved by several conclusions that we present below. The subdifferential $D^-\beta(p/q)$ is an interval $[c^-,c^+]$. For the cohomology class $c^+$, the Aubry set $\mathscr A(c^+)$ consists of the periodic orbits $\tilde\gamma_{p/q}$ with rotation number $p/q$ and all the forward heteroclinic orbits $\tilde\xi_{p/q}$ (see \cite[Section 8]{Mather1993}).

\subsubsection{Idea for proving Theorem \ref{thm:exp-flat-all}}
We start by giving a brief idea of the proof of Theorem \ref{thm:exp-flat-all}. To show the way to construct a long-range closed curve, we state a concrete and simplified example. The formal argument for the general case will be given in the next subsection.  
 
  Consider an example with $p/q=3/5$ and $\#(\mathscr M_{p/q})=1$ (Figure \ref{fig:concatenation}). First we pick an arbitrary positive integer $T$, and consider the case of $\delta=1/(2Tq)$. According to Proposition \ref{pro:structure}, we shall construct $\tilde\zeta$ by choosing segments of appropriate heteroclinic orbits $\tilde\xi_{p/q}^k$ between each two nearby lifts of periodic orbits with time interval $[(2k-3)T,(2k-1)T]$ for $k=1,\cdots,q$, so that we can concatenate them to form a loop on torus after the exponentially small deformation at both ends of each $\tilde\xi_{p/q}^k$. In this way, the rotation number of $\tilde\zeta$ is $p/q+\delta$. 
   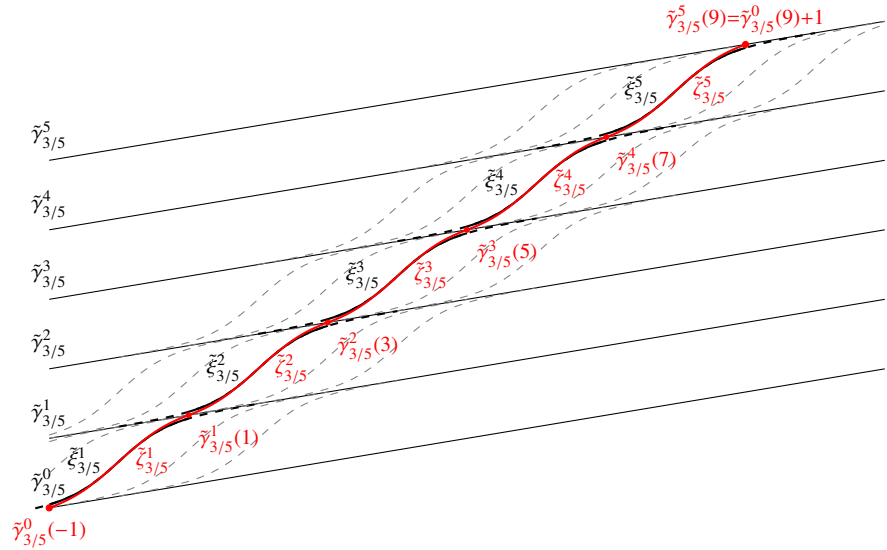
\begin{figure}[h]\label{fig:concatenation}
	\begin{center}
		\begin{tikzpicture}[scale=0.92]
		    \draw [] (-6,-3) .. controls (0,-2) .. (6,-1);
			\draw [] (-6,-4) .. controls (0,-3) .. (6,-2);
			\draw [] (-6,-5) .. controls (0,-4) .. (6,-3);
			\draw [] (-6,-6) .. controls (0,-5) .. (6,-4);
			\draw [] (-6,-7) .. controls (0,-6) .. (6,-5);
			\draw [] (-6,-8) .. controls (0,-7) .. (6,-6);
			\draw  (-6,-3) node[above] {$\scriptstyle\tilde\gamma_{3/5}^5$}; 
			\draw  (-6,-4) node[above] {$\scriptstyle\tilde\gamma_{3/5}^4$}; 
			\draw  (-6,-5) node[above] {$\scriptstyle\tilde\gamma_{3/5}^3$};
			\draw  (-6,-6) node[above] {$\scriptstyle\tilde\gamma_{3/5}^2$};
			\draw  (-6,-7) node[above] {$\scriptstyle\tilde\gamma_{3/5}^1$};
			\draw  (-6,-8) node[above] {$\scriptstyle\tilde\gamma_{3/5}^0$};
			\draw [domain=-6:-4,black,thick] plot (\x,{1/6*\x-7+1/(1+20^(-\x-5))});
			\draw [domain=-6.2:-3,black,dashed,thick] plot (\x,{1/6*\x-7+1/(1+20^(-\x-5))});
			\draw [domain=-4:-2,black,thick] plot (\x,{1/6*\x-6+1/(1+20^(-\x-3))});
			\draw [domain=-5:-1,black,dashed,thick] plot (\x,{1/6*\x-6+1/(1+20^(-\x-3))});
			\draw [domain=-2:0,black,thick] plot (\x,{1/6*\x-5+1/(1+20^(-\x-1))});
			\draw [domain=-3:1,black,dashed,thick] plot (\x,{1/6*\x-5+1/(1+20^(-\x-1))});
			\draw [domain=0:2,black,thick] plot (\x,{1/6*\x-4+1/(1+20^(-\x+1))});
			\draw [domain=-1:3,black,dashed,thick] plot (\x,{1/6*\x-4+1/(1+20^(-\x+1))});
			\draw [domain=2:4,black,thick] plot (\x,{1/6*\x-3+1/(1+20^(-\x+3))});
			\draw [domain=1:5,black,dashed,thick] plot (\x,{1/6*\x-3+1/(1+20^(-\x+3))});
			\draw [domain=-6:-4,gray,dashed] plot (\x,{1/6*\x-7+1/(1+20^(-\x-7))});
			\draw [domain=-6:-3,gray,dashed] plot (\x,{1/6*\x-7+1/(1+20^(-\x-6))});
			\draw [domain=-6:-2,gray,dashed] plot (\x,{1/6*\x-7+1/(1+20^(-\x-4))});
			\draw [domain=-5:-1,gray,dashed] plot (\x,{1/6*\x-7+1/(1+20^(-\x-3))});
			\draw [domain=-6:-3,gray,dashed] plot (\x,{1/6*\x-6+1/(1+20^(-\x-5))});
			\draw [domain=-6:-2,gray,dashed] plot (\x,{1/6*\x-6+1/(1+20^(-\x-4))});
			\draw [domain=-4:0,gray,dashed] plot (\x,{1/6*\x-6+1/(1+20^(-\x-2))});
			\draw [domain=-3:1,gray,dashed] plot (\x,{1/6*\x-6+1/(1+20^(-\x-1))});
			\draw [domain=-5:-1,gray,dashed] plot (\x,{1/6*\x-5+1/(1+20^(-\x-3))});
			\draw [domain=-4:0,gray,dashed] plot (\x,{1/6*\x-5+1/(1+20^(-\x-2))});
			\draw [domain=-2:2,gray,dashed] plot (\x,{1/6*\x-5+1/(1+20^(-\x))});
			\draw [domain=-1:3,gray,dashed] plot (\x,{1/6*\x-5+1/(1+20^(-\x+1))});
			\draw [domain=-3:1,gray,dashed] plot (\x,{1/6*\x-4+1/(1+20^(-\x-1))});
			\draw [domain=-2:2,gray,dashed] plot (\x,{1/6*\x-4+1/(1+20^(-\x))});
			\draw [domain=0:4,gray,dashed] plot (\x,{1/6*\x-4+1/(1+20^(-\x+2))});
			\draw [domain=1:5,gray,dashed] plot (\x,{1/6*\x-4+1/(1+20^(-\x+3))});
			\draw [domain=-1:3,gray,dashed] plot (\x,{1/6*\x-3+1/(1+20^(-\x+1))});
			\draw [domain=0:4,gray,dashed] plot (\x,{1/6*\x-3+1/(1+20^(-\x+2))});
			\draw [domain=2:6,gray,dashed] plot (\x,{1/6*\x-3+1/(1+20^(-\x+4))});
			\draw [domain=3:6,gray,dashed] plot (\x,{1/6*\x-3+1/(1+20^(-\x+5))});
            \draw [domain=-6+2/3:-4-2/3,red,thick] plot (\x,{1/6*\x-7+1/(1+20^(-\x-5))});%
			\draw [domain=-6:-6+2/3,red,thick] plot (\x,{1/6*\x-7+1/(1+20^(-\x-5))+3/2*(\x+6-2/3)/21});
			\draw [domain=-4-2/3:-4,red,thick] plot (\x,{1/6*\x-7+1/(1+20^(-\x-5))-3/2*(-\x-4-2/3)/21});

            \draw [domain=-4+2/3:-2-2/3,red,thick] plot (\x,{1/6*\x-6+1/(1+20^(-\x-3))});
            \draw [domain=-4:-4+2/3,red,thick] plot (\x,{1/6*\x-6+1/(1+20^(-\x-3))+3/2*(\x+4-2/3)/21});
			\draw [domain=-2-2/3:-2,red,thick] plot (\x,{1/6*\x-6+1/(1+20^(-\x-3))-3/2*(-\x-2-2/3)/21});
			
            \draw [domain=-2+2/3:0-2/3,red,thick] plot (\x,{1/6*\x-5+1/(1+20^(-\x-1))});
            \draw [domain=-2:-2+2/3,red,thick] plot (\x,{1/6*\x-5+1/(1+20^(-\x-1))+3/2*(\x+2-2/3)/21});
			\draw [domain=0-2/3:0,red,thick] plot (\x,{1/6*\x-5+1/(1+20^(-\x-1))-3/2*(-\x-2/3)/21});
			
            \draw [domain=+2/3:2-2/3,red,thick] plot (\x,{1/6*\x-4+1/(1+20^(-\x+1))});
            \draw [domain=0:2/3,red,thick] plot (\x,{1/6*\x-4+1/(1+20^(-\x+1))+3/2*(\x-2/3)/21});
			\draw [domain=2-2/3:2,red,thick] plot (\x,{1/6*\x-4+1/(1+20^(-\x+1))-3/2*(-\x+2-2/3)/21});
			
            \draw [domain=2+2/3:4-2/3,red,thick] plot (\x,{1/6*\x-3+1/(1+20^(-\x+3))});
            \draw [domain=2:2+2/3,red,thick] plot (\x,{1/6*\x-3+1/(1+20^(-\x+3))+3/2*(\x-2-2/3)/21});
			\draw [domain=4-2/3:4,red,thick] plot (\x,{1/6*\x-3+1/(1+20^(-\x+3))-3/2*(-\x+4-2/3)/21});
           
           \filldraw [red](-6,-8) circle (1.25pt) node[anchor=north] {$\scriptstyle\tilde\gamma_{3/5}^0(-1)$};
           \filldraw [red](-4,-20/3) circle (1pt) node[anchor=north west] {$\scriptstyle\tilde\gamma_{3/5}^1(1)$};
           \filldraw [red](-2,-16/3) circle (1pt) node[anchor=north west] {$\scriptstyle\tilde\gamma_{3/5}^2(3)$};
           \filldraw [red](0,-4) circle (1pt) node[anchor=north west] {$\scriptstyle\tilde\gamma_{3/5}^3(5)$};
           \filldraw [red](2,-8/3) circle (1pt) node[anchor=north west] {$\scriptstyle\tilde\gamma_{3/5}^4(7)$};
           \filldraw [red](4,-4/3) circle (1.25pt) node[anchor=south] {$\scriptstyle\tilde\gamma_{3/5}^5(9)=\tilde\gamma_{3/5}^0(9)+1$};
           
           \draw [red] (-4.9,-22/3) node[right] {$\scriptstyle\tilde\zeta_{3/5}^1$};
           \draw [black] (-5.1,-22/3) node[left] {$\scriptstyle\tilde\xi_{3/5}^1$};
           \draw [red] (-4.9+2,-6) node[right] {$\scriptstyle\tilde\zeta_{3/5}^2$};
           \draw [black] (-5.1+2,-6) node[left] {$\scriptstyle\tilde\xi_{3/5}^2$};
           \draw [red] (-4.9+4,-6+4/3) node[right] {$\scriptstyle\tilde\zeta_{3/5}^3$};
           \draw [black] (-5.1+4,-6+4/3) node[left] {$\scriptstyle\tilde\xi_{3/5}^3$};
           \draw [red] (-4.9+6,-6+8/3) node[right] {$\scriptstyle\tilde\zeta_{3/5}^4$};
           \draw [black] (-5.1+6,-6+8/3) node[left] {$\scriptstyle\tilde\xi_{3/5}^4$};
           \draw [red] (-4.9+8,-6+4) node[right] {$\scriptstyle\tilde\zeta_{3/5}^5$};
           \draw [black] (-5.1+8,-6+4) node[left] {$\scriptstyle\tilde\xi_{3/5}^5$};
           
		\end{tikzpicture}
	\end{center}
\caption{Construction of a loop by concatenating heteroclinic orbits ($p/q=3/5$, $T=1$, $\#(\mathscr M_{3/5})=1$): The black straight lines represent the lifts of the periodic orbit in $\mathscr M_{3/5}$ in $\R$. The thick red curve is formed by selecting specific segments (indicated by the thick black curves) from the gray dashed heteroclinic orbits and concatenating them after an exponentially small deformation.}
\end{figure}

Suppose $c^+$ is the cohomology class at the right endpoint of $D^-\beta(p/q)$. 
Let $\tilde\zeta_{p/q}^k$ be the concatenated segment of $\tilde\zeta$.  Then the exponentially small deformation implies $S_{c^+}(\tilde\zeta_{p/q}^k)=S_{c^+}(\tilde\xi_{p/q}^k)+O(e^{-\lambda T})$.
Also, 
    \begin{align*}
    \beta(p/q+\delta)&\leqslant \frac{1}{2Tq}S(\tilde\zeta)=\frac{1}{2qT}\sum_{k=1}^qS(\tilde\zeta_{p/q}^k)\\
    &=\frac{1}{2qT}\sum_{k=1}^qS_{c^+}(\tilde\xi_{p/q}^k|_{((2k-3)T,(2k-1)T)})+c^+(p/q+\delta)-\alpha(c^+)+O(e^{-\lambda T})\\
    &=\frac{1}{2qT}\sum_{k=1}^qS_{c^+}(\tilde\xi_{p/q}^k|_{((2k-3)T,(2k-1)T)})+c^+\delta+\beta(p/q)+O(e^{-\lambda T}).
    \end{align*}

In addition, the \(c^+\)-staticity of the heteroclinic orbits and the cyclic matching of their limiting periodic endpoints imply
$$
\lim_{n\to\infty}
\sum_{k=1}^{q}
S_{c^+}(
\tilde\xi_{p/q}^{k}
|_{[(2k-3)T-nq,\,(2k-1)T+nq]}
)=0.
$$
Comparing each of the two additional tails with the corresponding asymptotic periodic orbit, whose \(c^+\)-action over an integer multiple of \(q\) vanishes, and applying the exponential estimates in Proposition 3.8, we obtain
$$
\left|
\sum_{k=1}^{q}
S_{c^+}(
\tilde\xi_{p/q}^{k}
|_{[(2k-3)T,\,(2k-1)T]}
)
\right|
=O\left(qe^{-\lambda T}\right).
$$
The cancellation and the tail estimates are justified in the proof of Lemma \ref{pro:exp-flat-rational-integer} below. Together with the preceding inequality, this gives an exponential flatness estimate for \(\delta=1/(2Tq)\), \(T\in\mathbb N^*\). Convexity then extends the estimate to all sufficiently small \(\delta>0\), as detailed in the next subsection.

\subsubsection{Proof of Theorem \ref{thm:exp-flat-all}}
The following lemmas make the preceding argument rigorous. We set $t_{j,k}=2[k\#(\mathscr M_{p/q})+(j-1)]T$ for $k=0,\cdots,q-1$ and $j=1,\cdots,\#(\mathscr M_{p/q})$ in Proposition \ref{pro:structure}. 
\begin{Lem}\label{lem:beta-delta-upperbound}
	Under the assumptions in Proposition \ref{pro:structure}, there are $C_1,\lambda>0$ depending on $L$, such that for any $\delta>0$ with $T:=\frac{1}{2q\#(\mathscr M_{p/q})\delta}\in\N$, 
	\begin{align}\label{eq:beta-upperbound}
	\beta(p/q+\delta)\leqslant \delta\sum_{k=0}^{q-1}\sum_{j=1}^{\#(\mathscr M_{p/q})}S(\tilde\xi_{p/q,j}^k|_{[t_{j,k}-T,t_{j,k}+T]})+C_1\#(\mathscr M_{p/q})q\delta e^{-\frac{\lambda}{2q\#(\mathscr M_{p/q})\delta}},
	\end{align}		
where $C_1$ can be chosen as $C_1:=2C_0\Lip(L,K)$.
\end{Lem}
\begin{proof}
	By Proposition \ref{pro:structure}, we obtain $\tilde\xi_{p/q,j}^k$ satisfying \eqref{eq:exp-convergence-1}-\eqref{eq:exp-convergence-2}. We define 
	a curve
	\begin{equation}\label{eq:zeta-k}
	\begin{aligned}
		&\tilde\zeta_{p/q}^k:[t_{1,k}-T,t_{\#(\mathscr M_{p/q}),k}+T]\to\R\\
		&\tilde\zeta_{p/q}^k(t):=\tilde\zeta_{p/q,j}^k(t),\,\,\,\,t\in[t_{j,k}-T,t_{j,k}+T],
	\end{aligned}
	\end{equation}   
	where            
	\begin{align}
		\label{eq:zeta-k}
		&\tilde\zeta_{p/q}^k(t):=\notag\\
		&\left\{
		\begin{array}{ll}
			\tilde\xi_{p/q,j}^k(t)-\frac{1}{\tau}(t-t_{j,k}+T-\tau)\left[\tilde\gamma_{p/q,j-1}^{k}(t_{j,k}-T)-\tilde\xi_{p/q,j}^k(t_{j,k}-T)\right],& t\in[t_{j,k}-T,t_{j,k}-T+\tau];\\
			\tilde\xi_{p/q,j}^k(t),&t\in(t_{j,k}-T+\tau,t_{j,k}+T-\tau];\\
			\tilde\xi_{p/q,j}^k(t)-\frac{1}{\tau}(t_{j,k}+T-\tau-t)\left[\tilde\gamma_{p/q,j}^k(t_{j,k}+T)-\tilde\xi_{p/q,j}^k(t_{j,k}+T)\right],&t\in(t_{j,k}+T-\tau,t_{j,k}+T], \\
			&j\neq \#(\mathscr M_{p/q});\\
			\tilde\xi_{p/q,j}^k(t)-\frac{1}{\tau}(t_{j,k}+T-\tau-t)\left[\tilde\gamma_{p/q,0}^{k+1}(t_{j,k}+T)-\tilde\xi_{p/q,j}^k(t_{j,k}+T)\right],&t\in(t_{j,k}+T-\tau,t_{j,k}+T],\\
			 &j=\#(\mathscr M_{p/q})\\
		\end{array}
		\right.
	\end{align}   
	for small $\tau\in(0,1]$. Note that $\tilde\zeta_{p/q}^k(t_{\#(\mathscr M_{p/q}),k}+T)=\tilde\gamma_{p/q,0}^{k+1}(t_{\#(\mathscr M_{p/q}),k}+T)=\tilde\zeta_{p/q}^{k+1}(t_{1,k+1}-T)$. Therefore, we can connect the above curves as $\tilde\zeta_{p/q}:[-T,(2q\#(\mathscr M_{p/q})-1)T]\to\R$
	\begin{align*}
		\tilde\zeta_{p/q}(t):=\tilde\zeta_{p/q}^k(t),\,\,\,\,t\in[t_{1,k}-T,t_{\#(\mathscr M_{p/q}),k}+T], \,\,\,\,\,\,k=0,1,\cdots,q-1. 
	\end{align*}
	In addition, 
	\begin{align*}
		\tilde\zeta_{p/q}((2q\#(\mathscr M_{p/q})-1)T)&=\tilde\zeta_{p/q}^{q-1}((2q\#(\mathscr M_{p/q})-1)T)=\tilde\gamma_{p/q,0}^{q}((2q\#(\mathscr M_{p/q})-1)T)\\
		&=\tilde\gamma_{p/q,0}^{0}((2q\#(\mathscr M_{p/q})-1)T)+1=\tilde\gamma_{p/q,0}^{0}(-T)+2\#(\mathscr M_{p/q})Tp+1\\
		&=\tilde\zeta_{p/q}(-T)+2\#(\mathscr M_{p/q})Tp+1.
	\end{align*}
	which means that the homology class of $\zeta_{p/q}$, the projection of $\tilde\zeta_{p/q}$ on $\T$, is 
	\begin{align*}
		\left[\zeta_{p/q}\right]=\frac{2\#(\mathscr M_{p/q})Tp+1}{2\#(\mathscr M_{p/q})Tq}=\frac{p}{q}+\delta.
	\end{align*}
	Therefore \eqref{eq:beta-2} gives that
 \begin{align}\label{eq:beta-xi-zeta}
 	\beta(p/q+\delta)\leqslant \frac{S(\tilde\zeta_{p/q})}{2\#(\mathscr M_{p/q})Tq}=\frac{1}{2\#(\mathscr M_{p/q})Tq}\sum_{k=0}^{q-1}S(\tilde\zeta_{p/q}^k).
 \end{align}
Next we will give an upper bound of $S(\tilde\zeta_{p/q}^k)$. We rewrite 
	\begin{align}
		S(\tilde\zeta_{p/q}^k)=\sum_{j=1}^{\#(\mathscr M_{p/q})}\left[S(\tilde\xi_{p/q,j}^k|_{[t_{j,k}-T,t_{j,k}+T]})+S(\tilde\zeta_{p/q,j}^k)-S(\tilde\xi_{p/q,j}^k|_{[t_{j,k}-T,t_{j,k}+T]})\right].\label{eq:action-zeta-i}
	\end{align}
It is easy to calculate according to \eqref{eq:zeta-k} that, on $[t_{j,k}-T,t_{j,k}+T]$
	\begin{align}
		S(\tilde\zeta_{p/q,j}^k)-S(\tilde\xi_{p/q,j}^k)&=\int_{t_{j,k}-T}^{t_{j,k}+T}L\left(t,\tilde\zeta_{p/q}^k(t),\dot{\tilde\zeta}_{p/q}^k(t)\right)-L\left(t,\tilde\xi_{p/q}^k(t),\dot{\tilde\xi}_{p/q}^k(t)\right)\,dt\notag\\
		&=\int_{t_{j,k}-T}^{t_{j,k}-T+\tau}L\left(t,\tilde\zeta_{p/q,j}^k(t),\dot{\tilde\zeta}_{p/q,j}^k(t)\right)-L\left(t,\tilde\xi_{p/q,j}^k(t),\dot{\tilde\xi}_{p/q,j}^k(t)\right)\,dt\notag\\
		&\,\,\,\,+\int_{t_{j,k}+T-\tau}^{t_{j,k}+T}L\left(t,\tilde\zeta_{p/q,j}^k(t),\dot{\tilde\zeta}_{p/q}^k(t)\right)-L\left(t,\tilde\xi_{p/q,j}^k(t),\dot{\tilde\xi}_{p/q,j}^k(t)\right)\,dt=:\text{I}_{j,k}^{-}+\text{I}_{j,k}^{+}\label{eq:I_i+-}
	\end{align}
for each $k$. As for $\text{I}_{j,k}^-$, following from \eqref{eq:zeta-k} and \eqref{eq:exp-convergence-1},
    \begin{align}
		|\text{I}_{j,k}^{-}|&\leqslant\int_{t_{j,k}-T}^{t_{j,k}-T+\tau}\left|\frac{\Lip(L,K)(t-t_{j,k}+T-\tau)}{\tau}\left(\tilde\gamma_{p/q,j-1}^{k}(t_{j,k}-T)-\tilde\xi_{p/q,j}^k(t_{j,k}-T)\right)\right|\,dt\notag\\
		&\leqslant\Lip(L,K)C_0e^{-\lambda T}.\label{eq:I_k-}
	\end{align}
Likewise, we can also get whether $j$ equals $\#(M_{p/q})$ or not,
\begin{align}\label{eq:I_k+}
	|\text{I}_{j,k}^{+}|&\leqslant\Lip(L,K)C_0e^{-\lambda T}.
\end{align}
Consequently, substituting $T=\frac{1}{2q\#(\mathscr M_{p/q})\delta}$, we obtain by \eqref{eq:beta-xi-zeta}-\eqref{eq:I_k+} that 
	\begin{align*}
		\beta(p/q+\delta)&\leqslant \delta\sum_{k=0}^{q-1}\sum_{j=1}^{\#(\mathscr M_{p/q})}\left[S(\tilde\xi_{p/q,j}^k|_{[t_{j,k}-T,t_{j,k}+T]})+I_{j,k}^-+I_{j,k}^+\right]\\
		&\leqslant\delta\sum_{k=0}^{q-1}\sum_{j=1}^{\#(\mathscr M_{p/q})}S(\tilde\xi_{p/q,j}^k|_{[t_{j,k}-T,t_{j,k}+T]})+2\Lip(L,K)C_0q\#(\mathscr M_{p/q})\delta e^{-\frac{\lambda}{2q\#(\mathscr M_{p/q})\delta}}.
	\end{align*}
Letting $C_1:=2\Lip(L,K)C_0$, we have \eqref{eq:beta-upperbound}.
\end{proof}

	\begin{Lem}\label{pro:exp-flat-rational-integer}
	Under the assumptions in Lemma \ref{lem:beta-delta-upperbound}, 
	\begin{align}\label{eq:exp-flat}
	\beta(p/q+\delta)-\beta(p/q)-\beta_+'(p/q)\delta\leqslant C_2q\#(\mathscr M_{p/q})\delta e^{-\frac{\lambda}{2q\#(\mathscr M_{p/q})\delta}},
\end{align}	
where $C_2:=2(\Lip(L,K)/\lambda+1)C_0+C_1$.  
\end{Lem}
\begin{proof}
First, the convexity of $\beta$ yields that (see Appendix \ref{sec:convex}) 
\begin{align*}
	D^-\beta(p/q)=\left[\beta_-'(p/q),\beta_+'(p/q)\right]=\left[c^-,c^+\right],
\end{align*}where $c^+=\beta_+'(p/q)$ is the cohomology such that $\tilde\gamma_{p/q,j}^k,\tilde\xi_{p/q,j}^k\subset\mathscr A(c^+)$ for all $k=0,\cdots,q-1$ and $j=1,\cdots,\#(\mathscr M_{p/q})$. This implies 
\begin{align}
	S_{c^+}(\tilde\gamma_{p/q,j}^{k}|_{[r,r+iq]})=\int_{r}^{r+iq} \left[L\left(t,\tilde\gamma_{p/q,j}^{k}(t),\dot{\tilde\gamma}_{p/q,j}^{k}(t)\right)-c^+\cdot\dot{\tilde\gamma}_{p/q,j}^{k}(t)+\alpha(c^+)\right]\,dt=0,\,\,\,\,\forall r\in\R,i\in\N\label{eq:aubry-1}
\end{align}and by Definition \ref{defn:Aubry},
\begin{equation}\label{eq:aubry-2}
\begin{aligned}
   S_{c^+}(\tilde\xi_{p/q,j}^k|_{[t_{j,k}-T-nq,t_{j,k}+T+nq]})&=\Psi_{c^+}(\pi\tilde\xi_{p/q,j}^k(t_{j,k}-T-nq),\pi\tilde\xi_{p/q,j}^k(t_{j,k}+T+nq))\\
   &=-\Psi_{c^+}(\pi\tilde\xi_{p/q,j}^k(t_{j,k}+T+nq),\pi\tilde\xi_{p/q,j}^k(t_{j,k}-T-nq)).
\end{aligned}
\end{equation} 
Here $\pi:\R\to\T$ denotes the canonical projection. Since $t_{j,k}\pm T\in\Z$, all endpoints have time coordinate $0\in\T=\R/\Z$. We therefore use the shorthand $\Psi_c(x,y)=\Psi_c((0,x),(0,y))$.    
As a result of \eqref{eq:aubry-2}, we assert that
\begin{align}\label{eq:Sc-xi}
	\lim_{n\to\infty}\sum_{k=0}^{q-1}\sum_{j=1}^{\#(\mathscr M_{p/q})}S_{c^+}(\tilde\xi_{p/q,j}^k|_{[t_{j,k}-T-nq,t_{j,k}+T+nq]})=0.
\end{align}
In fact, by \eqref{eq:exp-convergence-1}-\eqref{eq:exp-convergence-2} and the periodicity of neighboring orbits, $\pi\tilde\xi_{p/q,j}^k(t_{j,k}-T-nq)\to x_{j,k}^-$ and $\pi\tilde\xi_{p/q,j}^k(t_{j,k}+T+nq)\to x_{j,k}^+$, where
\begin{align*}
	x_{j,k}^-:=\pi\gamma_{p/q,j-1}^k(t_{j,k}-T),\,\,\,\,\,\, x_{j,k}^+:=
	\left\{
	\begin{array}{ll}
		\pi\gamma_{p/q,j}^k(t_{j,k}+T),& j=1,\cdots,\#(\mathscr M_{p/q})-1,\\
		\pi\gamma_{p/q,0}^{k+1}(t_{j,k}+T),& j=\#(\mathscr M_{p/q}).
	\end{array}
	\right.
\end{align*}
These are the junctions where $\tilde\zeta$ is obtained by connecting curve $\tilde\zeta_{p/q}^k$. We also have 
\begin{align*}
	\left\{
	\begin{array}{ll}
		x_{j,k}^+=x_{j+1,k}^-,&j=1,\cdots,\#(\mathscr M_{p/q})-1;\\
		x_{\#(\mathscr M_{p/q}),k}^+=x_{1,k+1}^-,&k=0,\cdots,q-2;\\
		x_{\#(\mathscr M_{p/q}),q-1}^+=x_{1,0}^-.&
	\end{array}
	\right.
\end{align*}
Combining with \eqref{eq:aubry-2} and the continuity of Ma\~n\'e potential, we can obtain
\begin{align*}
	\lim_{n\to\infty}S_{c^+}(\tilde\xi_{p/q,j}^k|_{[t_{j,k}-T-nq,t_{j,k}+T+nq]})=\Psi_{c^+}(x_{j,k}^-,x_{j,k}^+),
\end{align*}
as well as $\Psi_{c^+}(x_{j,k}^-,x_{j,k}^+)=-\Psi_{c^+}(x_{j,k}^+,x_{j,k}^-)$. 
By the triangle inequality of Ma\~n\'e potential,
\begin{align}\label{eq:mane-0-1}
	0=\Psi_{c^+}(x_{1,0}^-,x_{1,0}^-)\leqslant\sum_{k=0}^{q-1}\sum_{j=1}^{\#(\mathscr M_{p/q})}\Psi_{c^+}(x_{j,k}^-,x_{j,k}^+).
\end{align}
Conversely, 
\begin{align}\label{eq:mane-0-2}
	0=\Psi_{c^+}(x_{\#(\mathscr M_{p/q}),q-1}^+,x_{\#(\mathscr M_{p/q}),q-1}^+)\leqslant \sum_{k=0}^{q-1}\sum_{j=1}^{\#(\mathscr M_{p/q})}\Psi_{c^+}(x_{j,k}^+,x_{j,k}^-)=-\sum_{k=0}^{q-1}\sum_{j=1}^{\#(\mathscr M_{p/q})}\Psi_{c^+}(x_{j,k}^-,x_{j,k}^+).
\end{align}
\eqref{eq:mane-0-1}-\eqref{eq:mane-0-2} give us $\sum_{k=0}^{q-1}\sum_{j=1}^{\#(\mathscr M_{p/q})}\Psi_{c^+}(x_{j,k}^-,x_{j,k}^+)=0$ and thus \eqref{eq:Sc-xi} holds true by \eqref{eq:aubry-2}.

Invoking \eqref{eq:exp-convergence-1}, \eqref{eq:exp-convergence-2} and \eqref{eq:aubry-1}, we have
\begin{align*}
	\left|S_{c^+}(\tilde\xi_{p/q,j}^k|_{(t_{j,k}+T,t_{j,k}+T+nq)})\right|&=\left|[S_{c^+}(\tilde\xi_{p/q,j}^k)-S_{c^+}(\tilde\gamma_{p/q,j}^k)]|_{(t_{j,k}+T,t_{j,k}+T+nq)}\right|\\
	&\leqslant\int_{t_{j,k}+T}^{t_{j,k}+T+nq}\left|L_{c^+}(t,\tilde\xi_{p/q,j}^k(t),\dot{\tilde\xi}_{p/q,j}^k(t))-L_{c^+}(t,\tilde\gamma_{p/q,j}^k(t),\dot{\tilde\gamma}_{p/q,j}^k(t))\right|\,dt\\
	&\leqslant \Lip(L,K)C_0/\lambda\left(e^{-\lambda T}-e^{-\lambda(T+nq)}\right)
\end{align*}
if $j\neq\#(\mathscr M_{p/q})$ and when $j=\#(\mathscr M_{p/q})$ the estimate of $\left|S_{c^+}(\tilde\xi_{p/q,j}^k|_{(t_{j,k}+T,t_{j,k}+T+nq)})\right|$ is similar. Here, $L_{c^+}(t,x,v):=L(t,x,v)-c^+v$. On the other hand,  
\begin{align*}
	\left|S_{c^+}(\tilde\xi_{p/q,j}^k|_{(t_{j,k}-T-nq,t_{j,k}-T)})\right|&=\left|[S_{c^+}(\tilde\xi_{p/q,j}^k)-S_{c^+}(\tilde\gamma_{p/q,j-1}^k)]|_{(t_{j,k}-T-nq,t_{j,k}-T)}\right|\\
	&\leqslant\int_{t_{j,k}-T-nq}^{t_{j,k}-T}\left|L_{c^+}(t,\tilde\xi_{p/q,j}^k(t),\dot{\tilde\xi}_{p/q,j}^k(t))-L_{c^+}(t,\tilde\gamma_{p/q,j-1}^{k}(t),\dot{\tilde\gamma}_{p/q,j-1}^{k}(t))\right|\,dt\\
	&\leqslant \sup_{c\in[c_1,c_2]}\Lip(L_c,K)C_0/\lambda\left(e^{-\lambda T}-e^{-\lambda(T+nq)}\right).
\end{align*}
Note that
\begin{align*}
	S_{c^+}&(\tilde\xi_{p/q,j}^k|_{[t_{j,k}-T,t_{j,k}+T]})\\
	&=S_{c^+}(\tilde\xi_{p/q,j}^k|_{[t_{j,k}-T-nq,t_{j,k}+T+nq]})-S_{c^+}(\tilde\xi_{p/q,j}^k|_{[t_{j,k}-T-nq,t_{j,k}-T]})-S_{c^+}(\tilde\xi_{p/q,j}^k|_{[t_{j,k}+T,t_{j,k}+T+nq]}).
\end{align*}
Therefore,
\begin{align}
	&\left|\sum_{k=0}^{q-1}\sum_{j=1}^{\#(\mathscr M_{p/q})}S_{c^+}(\tilde\xi_{p/q,j}^k|_{[t_{j,k}-T,t_{j,k}+T]})\right|\notag\\
	&\leqslant \left|\sum_{k=0}^{q-1}\sum_{j=1}^{\#(\mathscr M_{p/q})}S_{c^+}(\tilde\xi_{p/q,j}^k|_{[t_{j,k}-T-nq,t_{j,k}+T+nq]})\right|+2\sup_{c\in[c_1,c_2]}\Lip(L_c,K)C_0/\lambda \cdot q\#(\mathscr M_{p/q})e^{-\lambda T}
\end{align}
Letting $n\to\infty$, we can get 
\begin{align}
	\sum_{k=0}^{q-1}\sum_{j=1}^{\#(\mathscr M_{p/q})}S_{c^+}(\tilde\xi_{p/q,j}^k|_{[t_{j,k}-T,t_{j,k}+T]})&\leqslant \left|\sum_{k=0}^{q-1}\sum_{j=1}^{\#(\mathscr M_{p/q})}S_{c^+}(\tilde\xi_{p/q,j}^k|_{[t_{j,k}-T,t_{j,k}+T]})\right|\notag\\
	&\leqslant 2\sup_{c\in[c_1,c_2]}\Lip(L_c,K)C_0/\lambda \cdot q\#(\mathscr M_{p/q})e^{-\lambda T}\label{eq:sum-Sc-xi}. 
\end{align}

In addition, due to \eqref{eq:exp-convergence-1}-\eqref{eq:exp-convergence-2} and \eqref{eq:sum-Sc-xi},
\begin{align}
	&\,\,\,\,\,\,\sum_{k=0}^{q-1}\sum_{j=1}^{\#(\mathscr M_{p/q})}S(\tilde\xi_{p/q,j}^k|_{[t_{j,k}-T,t_{j,k}+T]})\notag\\
	&=\sum_{k=0}^{q-1}\sum_{j=1}^{\#(\mathscr M_{p/q})}\left[S_{c^+}(\tilde\xi_{p/q,j}^k|_{[t_{j,k}-T,t_{j,k}+T]})+\int_{t_{j,k}-T}^{t_{j,k}+T}c^+\cdot\dot{\tilde\xi}_{p/q,j}^k(t)\,dt\right]-2Tq\#(\mathscr M_{p/q})\alpha(c^+)\notag\\
	&\leqslant \sum_{k=0}^{q-1}\sum_{j=1}^{\#(\mathscr M_{p/q})}\int_{t_{j,k}-T}^{t_{j,k}+T}c^+\left[\dot{\tilde\zeta}_{p/q,j}^k(t)+\left(\dot{\tilde\xi}_{p/q,j}^k(t)-\dot{\tilde\zeta}_{p/q,j}^k(t)\right)\right]\,dt\notag\\
	&\qquad\qquad\qquad\qquad\qquad\qquad-2Tq\#(\mathscr M_{p/q})\alpha(c^+)+2\sup_{c\in[c_1,c_2]}\Lip(L_c,K)C_0/\lambda\cdot q\#(\mathscr M_{p/q})e^{-\lambda T}\notag\\
	&\leqslant c^+(2T\#(\mathscr M_{p/q})p+1)-2T\#(\mathscr M_{p/q})q\alpha(c^+)+2\left(\sup_{c\in[c_1,c_2]}\Lip(L_c,K)/\lambda+1\right)C_0q\#(\mathscr M_{p/q})e^{-\lambda T}.\label{eq:sum-S-xi} 
\end{align}
Then by Lemma \ref{lem:beta-delta-upperbound}, \eqref{eq:sum-S-xi} provides
\begin{align*}
	&\beta(p/q+\delta)\\
	&\leqslant \delta\sum_{k=0}^{q-1}\sum_{j=1}^{\#(\mathscr M_{p/q})}S(\tilde\xi_{p/q,j}^k|_{[t_{j,k}-T,t_{j,k}+T]})+C_1q\#(\mathscr M_{p/q})\delta e^{-\frac{\lambda}{2q\#(\mathscr M_{p/q})\delta}}\\
	&\leqslant \frac{1}{2Tq\#(\mathscr M_{p/q})}\left(c^+(2Tp\#(\mathscr M_{p/q})+1)-2Tq\#(\mathscr M_{p/q})\alpha(c^+)\right)\\
	&\qquad\qquad\qquad\qquad\qquad\qquad+\left[2\left(\sup_{c\in[c_1,c_2]}\Lip(L_c,K)/\lambda+1\right)C_0+C_1\right]q\#(\mathscr M_{p/q})\delta e^{-\frac{\lambda}{2q\#(\mathscr M_{p/q})\delta}}\\
	&=c^+(p/q+\delta)-\alpha(c^+)+\left[2\left(\sup_{c\in[c_1,c_2]}\Lip(L_c,K)/\lambda+1\right)C_0+C_1\right]q\#(\mathscr M_{p/q})\delta e^{-\frac{\lambda}{2q\#(\mathscr M_{p/q})\delta}}.
\end{align*}
Furthermore, since $c^+=\beta_+'(p/q)$, $\alpha(c^+)+\beta(p/q)=c^+\cdot p/q$ (see Appendix \ref{sec:convex}). It means
\begin{align*}
	\beta(p/q+\delta)\leqslant c^+\delta+\beta(p/q)+C_2q\#(\mathscr M_{p/q})\delta e^{-\frac{\lambda}{2q\#(\mathscr M_{p/q})\delta}},
\end{align*}
where $C_2:=2(\sup_{c\in[c_1,c_2]}\Lip(L_c,K)/\lambda+1)C_0+C_1$. 
\end{proof}

Based on Lemma \ref{pro:exp-flat-rational-integer}, we can prove Theorem \ref{thm:exp-flat-all} now. 
\begin{proof}[Proof of Theorem \ref{thm:exp-flat-all}]
	Now we consider the case of $T=\frac{1}{2q\#(\mathscr M_{p/q})\delta}>0$ (it does not have to be an integer) and denote $n:=\left\lfloor\frac{1}{2q\#(\mathscr M_{p/q})\delta}\right\rfloor$. Owing to the relation $x-1<\lfloor x\rfloor\leqslant x$ for all $x\in\R$, we have
	\begin{align}\label{eq:delta-n}
		\frac{1}{2q\#(\mathscr M_{p/q})\delta}-1<n\leqslant\frac{1}{2q\#(\mathscr M_{p/q})\delta}.
	\end{align}
	Let $\delta_2:=\frac{1}{2q\#(\mathscr M_{p/q})n}$ and $\delta_1:=\frac{1}{2q\#(\mathscr M_{p/q})(n+1)}$, then $\delta_1<\delta\leqslant \delta_2$.
	
	Given $h\in\R$, we set the function
	\begin{align*}
		\psi(t):=\beta(h+t)-\beta(h)-\beta_+'(h)t,\,\,\,\,t>0.
	\end{align*}
	Due to the convexity of $\beta$, $\psi$ is a convex function and according to Lemma \ref{pro:exp-flat-rational-integer},
	\begin{align*}
		\psi(\delta_1)\leqslant C_2q\#(\mathscr M_{p/q})\delta_1e^{-\frac{\lambda}{2q\#(\mathscr M_{p/q})\delta_1}},\,\,\,\,\psi(\delta_2)\leqslant C_2q\#(\mathscr M_{p/q})\delta_2e^{-\frac{\lambda}{2q\#(\mathscr M_{p/q})\delta_2}}.
	\end{align*}
	Now we can assert that if $\delta\in(\delta_1,\delta_2]$, then $\psi(\delta)\leqslant\max\{\psi(\delta_1),\psi(\delta_2)\}=\psi(\delta_2)$.
	First, from the convexity of $\psi$, we can easily observe that $\delta\in(\delta_1,\delta_2]$, $\psi(\delta)\leqslant\max\{\psi(\delta_1),\psi(\delta_2)\}$. In addition, notice that $f(x):=C_2q\#(\mathscr M_{p/q})xe^{-\frac{\lambda}{2q\#(\mathscr M_{p/q})x}}$ is monotone increasing, then $\max\{\psi(\delta_1),\psi(\delta_2)\}=\psi(\delta_2)$. Therefore,
	\begin{align}\label{eq:psi<f}
		\psi(\delta)\leqslant C_2q\#(\mathscr M_{p/q})\delta_2e^{-\frac{\lambda}{2q\#(\mathscr M_{p/q})\delta_2}}=f(\delta_2).
	\end{align}
	Following from \eqref{eq:delta-n}, if $\delta<\frac{1}{2q\#(\mathscr M_{p/q})}$,
	\begin{align*}
		\delta\leqslant\frac{1}{2q\#(\mathscr M_{p/q})n}=\delta_2<\frac{1}{2q\#(\mathscr M_{p/q})\left(\frac{1}{2q\#(\mathscr M_{p/q})\delta}-1\right)}=\frac{\delta}{1-2q\#(\mathscr M_{p/q})\delta}.
	\end{align*}
	By using the monotonicity of $f$ again, 
	\begin{align}
		f(\delta_2)&<f\left(\frac{\delta}{1-2q\#(\mathscr M_{p/q})\delta}\right)=C_2\frac{q\#(\mathscr M_{p/q})\delta}{1-2q\#(\mathscr M_{p/q})\delta}e^{-\frac{\lambda}{2q\#(\mathscr M_{p/q})}\cdot\frac{1-2q\#(\mathscr M_{p/q})\delta}{\delta}}\label{eq:f-upperbound-1}.
	\end{align}
	When $\delta\leqslant\frac{1}{4q\#(\mathscr M_{p/q})}$, we can get $\frac{1}{1-2q\#(\mathscr M_{p/q})\delta}\leqslant 2$, which also implies
	\begin{align}
		f(\delta_2)<2C_2q\#(\mathscr M_{p/q})\delta e^{-\frac{\lambda}{4q\#(\mathscr M_{p/q})\delta}}\label{eq:f-upperbound-2}.
	\end{align}
Invoking \eqref{eq:psi<f}-\eqref{eq:f-upperbound-2}, we get 
\begin{align*}
	\beta(p/q+\delta)-\beta(p/q)-\beta_+'(p/q)\delta\leqslant 2C_2\#(\mathscr M_{p/q})q\delta e^{-\frac{\lambda}{4q\#(\mathscr M_{p/q})\delta}}
\end{align*}
for $\delta\leqslant\frac{1}{4q\#(\mathscr M_{p/q})}$. In view of Proposition \ref{pro:uni-num-bounded}, $\#(\mathscr M_{p/q})\leqslant N_{[c_1,c_2]}$, which indicates when $\delta\leqslant \frac{1}{4qN_{[c_1,c_2]}}=:\varepsilon$, $\beta$ satisfies \eqref{eq:exp-flat-all} after choosing 
\begin{align*}
	C:=2C_2N_{[c_1,c_2]}=4\left(\sup_{c\in[c_1,c_2]}\Lip(L_c,K)(1+1/\lambda)+1\right)C_0N_{[c_1,c_2]}.
\end{align*} Here we complete the proof. 
\end{proof}

\subsection{Aubry's lemma}

In this section, we give the proof of Lemma \ref{lem:var=0} following  (\cite{10.1007/BFb0093512,MR1219348}) and prove Lemma \ref{lem:haus-dim=0}.
\begin{proof}[Proof of Lemma \ref{lem:var=0}]
	Suppose $a\in(\R\setminus \Q)\cap (0,1]$. We first recall that for any $s>0$ and $\nu\in(0,1)$, such that there exists a sequence of rational numbers $\{p_i/q_i\}_{i\geqslant 1}\subset [0,1]$, fulfilling 
	\begin{align*}
		0<a-\frac{p_i}{q_i}<\frac{s}{q_i^{1+\nu}}.
	\end{align*}
	If we define the union of intervals
	\begin{align*}
		I(s,\nu,Q):=\bigcup_{q>Q,0<\frac{p}{q}\leqslant 1}\left(\frac{p}{q},\frac{p}{q}+\frac{s}{q^{1+\nu}}\right],\,\,\,\,Q>1,
	\end{align*}then we have $(\R\setminus \Q)\cap [0,1]\subset I(s,\nu,Q)$ for all $Q>1$ because of the arbitrariness of $a$. We can have the following estimate $\Var(\phi_+',(\R\setminus\Q)\cap(0,1])$, the total variation of $\phi_+'$ on $(\R\setminus\Q)\cap(0,1]$ by the fact that $\phi_+'$ is non-decreasing:
	\begin{equation}\label{eq:var-1}
	\begin{aligned}
		0&\leqslant \Var(\phi_+',(\R\setminus\Q)\cap(0,1])\\
		&\leqslant \Var(\phi_+',I(s,\nu,Q))\\
		&\leqslant \sum_{q>Q,0<\frac{p}{q}\leqslant 1}\Var\left(\phi_+',\left(\frac{p}{q},\frac{p}{q}+\frac{s}{q^{1+\nu}}\right]\right)\\
		&=\sum_{q>Q,0<\frac{p}{q}\leqslant 1}\phi_+'\left(\frac{p}{q}+\frac{s}{q^{1+\nu}}\right)-\phi_+'\left(\frac{p}{q}\right).
	\end{aligned}
	\end{equation}In addition, because $\phi_+'$ is non-decreasing,
	\begin{align*}
		\int_0^1\phi_+'\left(\frac{p}{q}+\frac{s}{q^{1+\nu}}\right)\,ds\leqslant q^{1+\nu}\left[\phi\left(\frac{p}{q}+\frac{1}{q^{1+\nu}}\right)-\phi\left(\frac{p}{q}\right)\right].
	\end{align*}
	So \eqref{eq:var-1} can be rewritten as
	\begin{align*}
		0\leqslant \Var(\phi_+',(\R\setminus\Q)\cap(0,1])\leqslant\sum_{q>Q,0<\frac{p}{q}\leqslant 1}q^{1+\nu}\left[\phi\left(\frac{p}{q}+\frac{1}{q^{1+\nu}}\right)-\phi\left(\frac{p}{q}\right)-\phi_+'\left(\frac{p}{q}\right)\frac{1}{q^{1+\nu}}\right].
	\end{align*}
	For fixed $\nu\in(0,1)$ and sufficiently large $Q>0$, when $q>Q$, $1/q^{1+\nu}<1/(4qN_{[c_1,c_2]})$. Then we obtain from \eqref{eq:u-flat} that 
	\begin{align*}
		0\leqslant \Var(\phi_+',(\R\setminus\Q)\cap(0,1])\leqslant\sum_{q>Q,0<\frac{p}{q}\leqslant 1}q^{1+\nu}u_q\left(\frac{1}{q^{1+\nu}}\right)\leqslant\sum_{q>Q}q^{2+\nu}u_q\left(\frac{1}{q^{1+\nu}}\right).
	\end{align*}
Combining with \eqref{eq:u-sum-finite}, we can conclude that the rightmost item of the above inequality goes to 0 as $Q\to\infty$, which means $\Var(\phi_+',(\R\setminus\Q)\cap(0,1])=0$. 

Since $\phi$ is strictly convex, $\psi$ as the conjugate dual of $\phi$ is convex and of class $C^1$. In this case, $x=D\psi(y)$ if and only if $y\in D^-\phi(x)=[\phi_-'(x),\phi_+'(x)]$ (see Section \ref{sec:convex}). This implies that $D\psi\equiv x$ on $[\phi_-'(x),\phi_+'(x)]$ and then $D^2\psi\equiv 0$ on $(\phi_-'(x),\phi_+'(x))$. 
Define 
\begin{align*}
	E:=\bigcup_{D\psi([\phi_-'(x),\phi_+'(x)])=\{x\}\subset\Q\cap(0,1]}[\phi_-'(x),\phi_+'(x)]=\bigcup_{x\in\Q\cap(0,1]}[\phi_-'(x),\phi_+'(x)].
\end{align*}It means that $\mathscr L\left(E\right)=\Var(\phi_+',\Q\cap(0,1])=\phi_+'(1)-\phi_+'(0)$. In addition, we note that $D^-\phi(0)=[\phi_-'(0),\phi_+'(0)]$. One can get
\begin{align*}
	\mathscr L\left(\bigcup_{x\in\Q\cap[0,1]}[\phi_-'(x),\phi_+'(x)]\right)=\mathscr L\left(E\right)+\mathscr L\left(D^-\phi(0)\right)=\phi_+'(1)-\phi_-'(0).
\end{align*}
Equivalently, we obtain $D^2\psi=0$ $\mathscr L$-almost everywhere. 
\end{proof}
\begin{proof}[Proof of Lemma \ref{lem:haus-dim=0}]
According to the structure of $D\psi$, 
\begin{align*}
	Y:=\left\{y:D^2\psi(y)\text{ does not exist}\right\}\subset [\phi_-'(0),\phi_+'(1)]\setminus\bigcup_{x\in\Q\cap[0,1]}(\phi_-'(x),\phi_+'(x)).
\end{align*} 
Based on the argument at the proof of Lemma \ref{lem:var=0}, it can be known that
\begin{align}\label{eq:cover}
	Y\subset [\phi_-'(0),\phi_+'(1)]\setminus\bigcup_{x\in\Q\cap[0,1]}(\phi_-'(x),\phi_+'(x))\subset \bigcup_{\substack{q>Q\\0<\frac{p}{q}\leqslant 1}}\left[\phi_+'\left(\frac{p}{q}\right),\phi_+'\left(\frac{p}{q}+\frac{s}{q^{1+\nu}}\right)\right]
\end{align}for any fixed $s>0$, $\nu\in(0,1)$ and $Q>1$. It can be checked from \eqref{eq:u-flat} that for fixed $s_0\in(0,2)$,
\begin{align*}
	(2-s_0)\left[\phi_+'\left(\frac{p}{q}+\frac{s_0}{q^{1+\nu}}\right)-\phi_+'\left(\frac{p}{q}\right)\right]&\leqslant \int_{s_0}^2\left[\phi_+'\left(\frac{p}{q}+\frac{s}{q^{1+\nu}}\right)-\phi_+'\left(\frac{p}{q}\right)\right]\,ds\\
	&\leqslant \int_{0}^2\left[\phi_+'\left(\frac{p}{q}+\frac{s}{q^{1+\nu}}\right)-\phi_+'\left(\frac{p}{q}\right)\right]\,ds<q^{1+\nu}u_q\left(\frac{2}{q^{1+\nu}}\right).
\end{align*}As a result, 
\begin{align*}
	\phi_+'\left(\frac{p}{q}+\frac{s_0}{q^{1+\nu}}\right)-\phi_+'\left(\frac{p}{q}\right)<\frac{q^{1+\nu}}{2-s_0}u_q\left(\frac{2}{q^{1+\nu}}\right). 
\end{align*}
By assumptions on $u_q$, given $s\in(0,1)$, for any $\delta>0$, there exists $Q_\delta\in\N^*$ depending on $\delta$, such that the diameter of the each member of the cover of $Y$ at the rightmost of \eqref{eq:cover} is less than $\delta$. Then for $\theta\in(0,1]$ and $s\in(0,1)$, 
\begin{align}
	\mathscr H_\delta^\theta\left(Y\right):=\inf\left\{\sum_{i=1}^\infty(\mathrm{diam}\,  C_i)^\theta:Y\subset \bigcup_{i=1}^\infty C_i,\,\mathrm{diam}\,  C_i<\delta \right\}\leqslant \sum_{q>Q_\delta, 0<\frac{p}{q}\leqslant 1}\left[\phi_+'\left(\frac{p}{q}+\frac{s}{q^{1+\nu}}\right)-\phi_+'\left(\frac{p}{q}\right)\right]^\theta. 
\end{align}
  Integrating both sides of the inequality above with respect to $s$ over $[0,1]$, we can get 
  \begin{align*}
  	\mathscr H_\delta^\theta\left(Y\right)\leqslant \sum_{q>Q_\delta, 0<\frac{p}{q}\leqslant 1}\int_0^1\left[\phi_+'\left(\frac{p}{q}+\frac{s}{q^{1+\nu}}\right)-\phi_+'\left(\frac{p}{q}\right)\right]^\theta\,ds. 
  \end{align*}
  by monotone convergence theorem. Since the function $x\mapsto x^\theta$ is concave and monotone increasing on $[0,+\infty)$ when $\theta\in(0,1]$, Jensen's inequality indicates that
  \begin{align}
  	\mathscr H_\delta^\theta\left(Y\right)&\leqslant \sum_{q>Q_\delta, 0<\frac{p}{q}\leqslant 1}\left[\int_0^1\phi_+'\left(\frac{p}{q}+\frac{s}{q^{1+\nu}}\right)\,ds-\phi_+'\left(\frac{p}{q}\right)\right]^\theta\notag\\
  	&\leqslant\sum_{q>Q_\delta, 0<\frac{p}{q}\leqslant 1}\left[q^{1+\nu}\left[\phi\left(\frac{p}{q}+\frac{1}{q^{1+\nu}}\right)-\phi\left(\frac{p}{q}\right)-\phi_+'\left(\frac{p}{q}\right)\frac{1}{q^{1+\nu}}\right]\right]^\theta\notag\\
  	&\leqslant \sum_{q>Q_\delta, 0<\frac{p}{q}\leqslant 1} q^{(1+\nu)\theta}u_q^\theta\left(\frac{1}{q^{1+\nu}}\right)=\sum_{q>Q_\delta} q^{(1+\nu)\theta+1}u_q^\theta\left(\frac{1}{q^{1+\nu}}\right).\label{eq:theta-convergence}
  \end{align} As $\delta\to 0^+$, we conclude by \eqref{eq:u-sum-finite-theta} and \eqref{eq:theta-convergence} that the $\theta$-dimensional Hausdorff outer measure of $Y$
  \begin{align*}
  	\mathscr H^\theta\left(Y\right):=\lim_{\delta\to 0^+}\mathscr H_\delta^\theta\left(Y\right)=0,\,\,\,\,\theta\in(0,1],
  \end{align*}which means the Hausdorff dimension of set $Y:=\{y:D^2\psi(y)\text{ does not exist}\}$ is 0. 
\end{proof}

\section{Incomplete devil's staircase in the presence of KAM tori}\label{SIncomplete}
In this section we show that the graph of $D\alpha$ cannot be a complete devil's staircase in the presence of KAM tori, and give the proof of Theorem \ref{ThmKAM}. 
\subsection{KAM torus}
\subsubsection{Some facts on KAM theory}\label{sec: kam}
Consider the nearly integrable Tonelli Hamiltonian $H:\T\times T^*\T^m\to\R$, which has the form
\begin{align}\label{eq:perturb-H}
	H(t,x,p):=H_0(p)+\varepsilon H_1(t,x,p). 
\end{align}
This is the small perturbation of an integrable system: when $\varepsilon=0$, the system is integrable for Hamiltonian $H_0$ and all orbits satisfy $\dot x=DH_0(p)=:\omega(p)$, thus are linear rotation $\gamma(t)=\gamma(0)+\omega(p)t$ with rotation vector $\rho(\gamma)=\omega(p)$ on $\T^m$. 

A vector (number) $\omega\in\R^m$ is called a ($(C,\tau)$-)\textit{Diophantine vector (number)}, if there exist $C>0$ and $\tau>0$ such that
\begin{align*}
	|\langle\omega,\mathbf k\rangle+l|\geqslant\frac{C}{|\mathbf k|^\tau}\,\,\,\,\forall \mathbf k\in\Z^m\setminus\{0\}, l\in\Z.
\end{align*}
The set of all $(C,\tau)$-Diophantine vectors (numbers) is denoted by $D(C,\tau)$. 

We state the following version of the KAM theorem.
\begin{The}
    Given rotation vector $\omega_0=DH_0(p_0)\in D(C,\tau)$, suppose that the Hessian is not degenerate at $p_0$: $\det(D^2H_0(p_0))\neq 0$, then 
there exists $\varepsilon_0>0$ such that when 
 $\|\varepsilon H_1\|_{C^r(\T\times T^*\T^m)}<\varepsilon_0$ for sufficiently large $r\in\N$ (depending on $m$)\label{it:ph1},
the Hamiltonian flow $\Phi_H$ admits an invariant torus $\mathcal T_{\omega_0}\subset \T\times T^*\T^m$ and $\Phi_H|_{\mathcal T_{\omega_0}}$ can be conjugated to the linear rotation with rotation vector $\omega_0$ on $\T^m$. More precisely, we can find a diffeomorphism $\Psi_{\omega_0}:\mathcal T_{\omega_0}\to \T\times\T^{m}(\times\{p_0\})$ fulfilling
\begin{align*}
	\Psi_{\omega_0}\circ\Phi_H^t|_{\mathcal T_{\omega_0}}\circ[\Psi_{\omega_0}]^{-1}(s,x)=(s+t,x+t\omega_0),\,\,\,\,\forall (s,x)\in\T\times\T^m,\,\,t\in\R.
\end{align*}
\end{The}
Such invariant torus $\mathcal T_{\omega_0}$ is called a \textit{KAM torus} with rotation vector $\omega_0$. Recall that for Tonelli Hamiltonian systems, each KAM torus is the Mather set associated to some corresponding cohomology, which means $\mathcal T_{\omega_0}$ is a Lipschitz graph. This implies that the inverse of canonical projection $\pi|_{\mathcal T_{\omega_0}}^{-1}: \T\times \T^m\to \mathcal T_{\omega_0}\subset \T\times T^*\T^m$ is well-defined. Consequently, $\varphi_{\omega_0}:=\Psi_{\omega_0}\circ\pi|_{\mathcal T_{\omega_0}}^{-1}$ is a diffeomorphism of $\T\times\T^m$.

Given a Diophantine vector $h$ and assume that $\tilde{\mathscr M}_h$ is a KAM torus. Let $\hat H_h:=\varphi_h^*H$ be the pull-back of Hamiltonian $H$  of the canonical transformation $\varphi_h$ and $\hat L_h$ be the Lagrangian associated to $\hat H_h$. 
Since for $x\in\T^m$, $\gamma_x(t):=x+th$ is the solution of Euler-Lagrange flow $\Phi_{\hat L_h}$. It indicates that
 \begin{align*}
 	\frac{d}{dt}\left[\frac{\partial \hat L_h}{\partial v}(t,x+th,h)\right]=\frac{\partial \hat L_h}{\partial x}(t,x+th,h).
 \end{align*}
Due to the fact that $h$ is Diophantine, $\gamma_x$ is dense in $\T^m$, which means on KAM torus with rotation vector $h$, 
 \begin{enumerate}[\rm (1)]
 	
 	\item $\frac{\partial \hat L_h}{\partial v}(\cdot,\cdot,h)$ is a constant map from $\T\times\T^m$ to $\R^m$, which is denoted by $\eta_h\in\R^m$;
 	\item $\hat L_h(t,\cdot,h)$ is a constant and we denote by $\hat L_h(t,h)$ for short.
 \end{enumerate}
 Note that the conjugate transformation $\Psi_h$ in KAM theorem is obtained by a sequence of compositions of  exact symplectic transformations defined on smaller and smaller neighborhoods of the KAM torus. If we denote $\hat\beta_h$ by the the $\beta$-function associated to $\hat L_h$, then $\hat\beta_h=\beta$ due to Bernard (\cite{MR2309170}).  

 \subsubsection{KAM torus causes strong convexity of $\beta$-functions}In the following, we focus on the case $m=1$. 
\begin{The}\label{thm:strong-convex}
	 For any $h\in D(C,\tau)$, there exist $\delta_h, c_{h}>0$, such that
	\begin{align}\label{eq:hatbeta}
		\beta(\rho)-\beta(h)-D\beta(h)(\rho-h)\geqslant c_{h}(\rho-h)^2,\,\,\,\, \rho\in B(h,\delta_h).
	\end{align}
\end{The}
\begin{proof}
	Since $\hat\beta_h=\beta$, we will focus on the analysis on Lagrangian $\hat L_h$. Because $\hat L_h$ has a positively definite Hessian with respect to component $v$,  there exist $C_{h}>c_{h}>0$, satisfying for all $|w|<1$, 
	\begin{align*}
		c_{h}|w|^2\leqslant \frac{\partial^2\hat L_h}{\partial v^2}(t,x,h+w)(w,w)\leqslant C_{h}|w|^2.
	\end{align*}
According to Taylor's expansion of $\hat L_h$,
\begin{align*}
	\hat L_h(t,x,h+w)=\hat L_h(t,h)+\eta_h(w)+\frac{1}{2}\frac{\partial^2\hat L}{\partial v^2}(t,x,h+\theta w)(w,w),\,\,\,\,\theta\in[0,1],
\end{align*}we have the following estimate
\begin{align}\label{eq:order2-esti}
\hat L_h(t,h)+\eta_h(w)+\frac{c_{h}}{2}|w|^2\leqslant\hat L_h(t,x,h+w)\leqslant\hat L_h(t,h)+\eta_h(w)+\frac{C_{h}}{2}|w|^2.
\end{align}
Assume that $p,q\in\Z$ and $\gamma:\R\to\T$ is a closed global minimal curve with rotation number $\rho(\gamma)=p/q$, and denote $\tilde\gamma$ as the lift of curve $\gamma$ from $\T$ to $\R$ as in the previous text. Then  $\tilde\gamma(t+q)=\tilde\gamma(t)+p$ for all $t\in\R$. By choosing the appropriate $p/q$, we let $p/q\in\Q$ satisfy $|p/q-h|<\delta$.


Since $d\gamma\subset\tilde {\mathscr N}(p/q)$, when $|p/q-h|<\delta(1,h)$, $\tilde{\mathscr N}(p/q)\subset B(\tilde{\mathscr N}(h),1)$  due to the upper-semicontinuity of Ma\~n\'e sets. The comparison of the average actions of $\tilde\gamma$ on $[0,q]$ can be obtained by \eqref{eq:order2-esti} from
\begin{align*}
	\frac{1}{q}\int_0^{q}\hat L_h(t,\tilde\gamma(t),\dot{\tilde\gamma}(t))\,dt&\geqslant\frac{1}{q}\int_0^{q}\left[\hat L_h(t,h)+\eta_h(\dot{\tilde\gamma}(t)-h)+\frac{c_{h}}{2}|\dot{\tilde\gamma}(t)-h|^2\right]\,dt\\
	&\geqslant\frac{1}{q}\int_0^{q}\hat L_h(t,h)\,dt+\eta_h\left(p/q-h\right)+\frac{c_{h}}{2}\left|p/q-h\right|^2\\
	&=\int_0^{1}\hat L_h(t,h)\,dt+\eta_h\left(p/q-h\right)+\frac{c_{h}}{2}\left|p/q-h\right|^2
\end{align*}
where the last inequality and equality hold by Jensen's inequality and the fact that $\hat L_h$ is an 1-time periodic Lagrangian. Taking minimum over all global minimal curves with rotation vector $p/q$ on the left hand side of the above inequality, we can get
\begin{align}\label{eq:hatbeta-}
	\hat\beta_h\left(p/q\right)\geqslant \int_0^{1}\hat L_h(t,h)\,dt+\eta_h\left(p/q-h\right)+\frac{c_{h}}{2}\left|p/q-h\right|^2=:\hat L_h^-\left(p/q\right).
\end{align}
In addition, after taking $\tilde{\gamma}_0(t):=x+p/qt$, we have for $|p/q-h|<\delta$,
\begin{align}\label{eq:hatbeta+}
	\hat\beta_h\left(p/q\right)\leqslant\frac{1}{q}S(\tilde\gamma_0|_{[0,q]})&=\frac{1}{q}\int_0^{q}\hat L_h\left(t,p/q\right)\,dt\notag\\
	&\leqslant \int_0^{1}\hat L_h(t,h)\,dt+\eta_h(p/q-h)+\frac{C_{h}}{2}\left|p/q-h\right|^2=:\hat L_h^+\left(p/q\right).
\end{align}
Together with \eqref{eq:hatbeta-}-\eqref{eq:hatbeta+} and the continuity of $\hat\beta$ and $\hat L^{\pm}$, we have 
\begin{align}\label{eq:hat-beta-control}
	\hat L_h^-(\rho)\leqslant\hat\beta_h(\rho)\leqslant\hat L_h^+(\rho),\,\,\,\,\forall |\rho-h|<\delta.
\end{align}
Let $\rho=h$, we have $\hat\beta_h(h)=\int_0^1\hat L_h(t,h)\,dt$. Furthermore, \eqref{eq:hat-beta-control} implies $D\hat\beta_h(h)=\eta_h$ and hence \eqref{eq:hatbeta} holds true by $\hat\beta_h=\beta$.
\end{proof}

\subsection{Incomplete devil's staircase}
The following result indicates that if a convex function has strong convexity on a certain positive measure set as in the previous theorem, then the differential of its convex conjugate as a continuous function has a non-trivial absolutely continuous part. Theorem \ref{ThmKAM} then follows. 
\begin{Pro}\label{pro:incomplete}
	 Suppose $\phi$ is a strictly convex function on $\R$ and $\psi$ is its convex conjugate. Given $P\subset\R$ with positive Lebesgue measure, if for any $x_0\in P$, there exist $c_{x_0},\delta_{x_0}>0$ such that when $x\in B(x_0,\delta_{x_0})$, 
	\begin{align}\label{eq:strong-convex}
		\phi(x)-\phi(x_0)-y_0(x-x_0)\geqslant c_{x_0}|x-x_0|^2,\,\,\,\,\forall y_0\in D^-\phi(x_0),
	\end{align}then $D\psi\in C(\R)$ (see Appendix \ref{sec:convex}) has a Lipschitz part on a positive Lebesgue measure set and thus is an incomplete devil's staircase. 
\end{Pro}
\begin{proof}
First we set $y_0\in (D\psi)^{-1}(P)$, then $x_0:=D\psi(y_0)\in P$ and $y_0\in D^-\phi(x_0)$. Due to the continuity of $D\psi$, we can find some $\varepsilon>0$ (depending on $x_0$ in general), such that if $|y-y_0|<\varepsilon$, $|D\psi(y)-D\psi(y_0)|<\delta_{x_0}$. Let $x=D\psi(y)$, then $x\in B(x_0,\delta_{x_0})$. In view of \eqref{eq:strong-convex}, we obtain
\begin{align}\label{eq:Dpsi-lip-1}
	c_{x_0}|x-x_0|^2\leqslant (y-y_0)(x-x_0)\leqslant|y-y_0||x-x_0|
\end{align}
by \eqref{eq:concave}. Recall that $D\psi(y)=x$, $D\psi(y_0)=x_0$. When $x\neq x_0$, \eqref{eq:Dpsi-lip-1} gives that for $y\in B(y_0,\varepsilon)$,
\begin{align*}
	|D\psi(y)-D\psi(y_0)|\leqslant\frac{1}{c_{x_0}}|y-y_0|.
\end{align*}
As a result, for any $y_0\in (D\psi)^{-1}(P)$,
\begin{align}\label{eq:stepanov}
	\limsup_{y\to y_0}\frac{|D\psi(y)-D\psi(y_0)|}{|y-y_0|}<\infty.
\end{align}

For $k\in\N^*$, define the set
\begin{align*}
	E_k:=\left\{y\in (D\psi)^{-1}(P):\forall y'\in (D\psi)^{-1}(P) \text{ with } 0<|y-y'|<\frac{1}{k}, \frac{|D\psi(y)-D\psi(y')|}{|y-y'|}\leqslant k \right\}.
\end{align*}
\eqref{eq:stepanov} indicates that $(D\psi)^{-1}(P)=\bigcup\limits_{k=1}^\infty E_k$. Thus, 
\begin{align*}
	P=D\psi\left((D\psi)^{-1}(P)\right)=\bigcup_{k=1}^\infty D\psi(E_k).
\end{align*}
Since $\mathscr L(P)>0$ by assumption, there exists at least one $k_0\in\N$, $\mathscr L(D\psi(E_{k_0}))>0$.

Moreover, $D\psi$ is (locally) Lipschitz on $E_k\subset (D\psi)^{-1}(P)$: when $|y'-y| <1/k$, $|D\psi(y')-D\psi(y)|\leqslant k|y'-y|$; if $|y'-y|\geqslant 1/k$, $|D\psi(y')-D\psi(y)|\leqslant 2B\leqslant 2Bk|y'-y|$, where $B$ is a local upper bound of $|D\psi|$. According to Rademacher theorem and the convexity of $\psi$, $D^2\psi\geqslant 0$ exists a.e. on $E_k$. 

In addition, area formula (see \cite{Evans_Gariepy_book}) implies that 
\begin{align*}
	\int_{E_{k_0}}D^2\psi(y)\,dy=\int_{\R}\mathcal \#\left(E_{k_0}\cap (D\psi)^{-1}(x)\right)\,dx\geqslant\mathscr L(D\psi(E_{k_0}))>0,
\end{align*}which means $D^2\psi$ does not vanish on a subset of $E_{k_0}$ with positive Lebesgue measure.
\end{proof}

\appendix 
\section{Some  convex analysis}\label{sec:convex}Some relevant basic notions about convex analysis are listed below. For the following and more, readers can refer to \cite{Cannarsa_Sinestrari_book,Hiriart-Urruty_Lemarechal_book2001}.
\begin{enumerate}[\textbullet]
	\item Let $I\subset\R$ be an open interval. A continuous function $\phi:I\to\R$ is called a \emph{convex function} if  for any $x\in I$ there exists $y\in\R$ such that
	\begin{equation}\label{eq:concave}
	    \phi(x')\geqslant\phi(x)+y(x'-x),\qquad \forall x'\in I,
	\end{equation}
	and we denote as $\phi\in\Cov(I)$. The set of $y$ satisfying \eqref{eq:concave} is called the \emph{subdifferential of $\phi$ at $x$} and we denote it by $D^-\phi(x)$.
	\item If $\phi\in\Cov(I)$,  then $\phi$ is locally Lipschitz on $I$, thus it is differentiable almost everywhere (in the sense of Lebesgue measure). Moreover, $D^-\phi(x)$ is a non-empty, compact and convex set for any $x\in\R$. $D^-\phi(x)$ is a singleton if and only if $\phi$ is differentiable at $x$. In this case, $D^-\phi(x)=\{D\phi(x)\}$. 
	\item The \textit{left} and \textit{right derivative }of $\phi$ 	\begin{align}\label{eq:direct-deriv}
		\phi_{-}'(x):=\lim_{t\to 0^+}\frac{\phi(x)-\phi(x-t)}{t},\,\,\,\,\phi_{+}'(x):=\lim_{t\to 0^+}\frac{\phi(x+t)-\phi(x)}{t}
	\end{align}exist at all $x\in I$ and $D^-\phi(x)=[\phi_-'(x),\phi_+'(x)]$.
	\item As a set-valued map, $x\rightsquigarrow D^-\phi(x)$ is upper-semicontinuous, i.e., both $x_n\to x$ and $D^-\phi(x_n)\ni y_n\to y$ as $n\to\infty$ imply $y\in D^-\phi(x)$. Therefore, if $\phi$ is differentiable on $I$, $\phi\in C^1(I)$.
	\item  For $\phi\in\Cov(I)$, its \textit{convex conjugate} 
	\begin{align*}
		\psi(y):=\sup\{yx-\phi(x):x\in I\}
	\end{align*} 
	is also convex and $y\in D^-\phi(x)$ if and only if $x\in D^-\psi(y)$ if only if $\phi(x)+\psi(y)=xy$. Furthermore, if $\phi$ is strictly convex, i.e., only the strict inequality holds in \eqref{eq:concave} when $x'\neq x$, then $D^-\psi$ is singleton everywhere and thus $\psi$ is of class $C^1$ on its domain.
\end{enumerate}
\section{Twist maps}\label{sec:twist map}
We provide the definition of twist maps here, though we do not use it in the main body of the paper. 
\begin{defn}[Twist map]
	A map $\Gamma=(\Gamma_x,\Gamma_p)\in C^1(T^*\T,T^*\T)$ is said to be an \textit{area-preserving exact twist map} (twist map for short), if it satisfies the following conditions
	\begin{enumerate}[\rm (1)]
		\item (area-preserving) $\Gamma^*\omega=\omega$, where $\omega:=dp\wedge dx$;
		\item (twist condition) there exists some $b>0$, $\frac{\partial}{\partial p}\Gamma_x(x,p)\geqslant b$;
		\item (exactness) $\Gamma^*\lambda_0-\lambda_0$ is exact, where $\lambda_0:=pdx$.
	\end{enumerate}
\end{defn}
\begin{The}
	There exists a generating function $h(x,x')$ satisfying
	\begin{enumerate}[\rm (1)]
		\item $h(x+1,x'+1)=h(x,x')$, $x\in\T$;
		\item $\partial_{12}^2h\leqslant -b^{-1}<0$;
		\item $dh(x,x')=p'dx'-pdx$, where $(x',p')=\Gamma(x,p)$.
	\end{enumerate}
\end{The}
\paragraph{\textbf{Data availability statement}}No data is involved in this article.
\paragraph{\textbf{Conflict of interest statement}}On behalf of all authors, the corresponding author states that there is no conflict of interest.

\bibliographystyle{plain}
\bibliography{mybib}

@inproceedings{aubry1978new,
  title={The new concept of transitions by breaking of analyticity in a crystallographic model},
  author={Aubry, Serge},
  booktitle={Solitons and Condensed Matter Physics: Proceedings of the Symposium on Nonlinear (Soliton) Structure and Dynamics in Condensed Matter, Oxford, England, June 27--29, 1978},
  pages={264--277},
  year={1978},
  organization={Springer}
}

@article {MR710121,
    AUTHOR = {Burkov, S. E. and Sinai, Ya. G.},
     TITLE = {Phase diagrams of one-dimensional lattice models with
              long-range antiferromagnetic interaction},
   JOURNAL = {Uspekhi Mat. Nauk},
  FJOURNAL = {Akademiya Nauk SSSR i Moskovskoe Matematicheskoe Obshchestvo.
              Uspekhi Matematicheskikh Nauk},
    VOLUME = {38},
      YEAR = {1983},
    NUMBER = {4(232)},
     PAGES = {205--225},
      ISSN = {0042-1316},
   MRCLASS = {82A68 (82A05 82A25)},
  MRNUMBER = {710121},
MRREVIEWER = {V.\ Z.\ Enol\cprime ski\u i},
}

@article {MR847308,
    AUTHOR = {Moser, J\"urgen},
     TITLE = {Minimal solutions of variational problems on a torus},
   JOURNAL = {Ann. Inst. H. Poincar\'e{} Anal. Non Lin\'eaire},
  FJOURNAL = {Annales de l'Institut Henri Poincar\'e. Analyse Non
              Lin\'eaire},
    VOLUME = {3},
      YEAR = {1986},
    NUMBER = {3},
     PAGES = {229--272},
      ISSN = {0294-1449},
   MRCLASS = {58E30 (58E15)},
  MRNUMBER = {847308},
MRREVIEWER = {Helmut\ Kaul},
       URL = {http://www.numdam.org/item?id=AIHPC_1986__3_3_229_0},
}

@article {MR991874,
    AUTHOR = {Bangert, Victor},
     TITLE = {On minimal laminations of the torus},
   JOURNAL = {Ann. Inst. H. Poincar\'e{} Anal. Non Lin\'eaire},
  FJOURNAL = {Annales de l'Institut Henri Poincar\'e. Analyse Non
              Lin\'eaire},
    VOLUME = {6},
      YEAR = {1989},
    NUMBER = {2},
     PAGES = {95--138},
      ISSN = {0294-1449},
   MRCLASS = {58E15 (49A21 58F18 58F27)},
  MRNUMBER = {991874},
MRREVIEWER = {Y.\ Mut\^o},
       URL = {http://www.numdam.org/item?id=AIHPC_1989__6_2_95_0},
}

@article {MR2197072,
    AUTHOR = {Auer, Franz and Bangert, Victor},
     TITLE = {Differentiability of the stable norm in codimension one},
   JOURNAL = {Amer. J. Math.},
  FJOURNAL = {American Journal of Mathematics},
    VOLUME = {128},
      YEAR = {2006},
    NUMBER = {1},
     PAGES = {215--238},
      ISSN = {0002-9327,1080-6377},
   MRCLASS = {49Q20 (53C20 53C65)},
  MRNUMBER = {2197072},
MRREVIEWER = {Giandomenico\ Orlandi},
       URL =
              {http://muse.jhu.edu/journals/american_journal_of_mathematics/v128/128.1auer.pdf},
}

@article{Bernard2008_1,
	author = {Bernard, Patrick},
	fjournal = {Journal of the American Mathematical Society},
	issn = {0894-0347},
	journal = {J. Amer. Math. Soc.},
	mrclass = {37J40 (37J50)},
	mrnumber = {2393423},
	mrreviewer = {Karl Friedrich Siburg},
	number = {3},
	pages = {615--669},
	title = {The dynamics of pseudographs in convex {H}amiltonian systems},
	url = {https://doi.org/10.1090/S0894-0347-08-00591-2},
	volume = {21},
	year = {2008}}

@book{Cannarsa_Sinestrari_book,
	author = {Cannarsa, Piermarco and Sinestrari, Carlo},
	isbn = {0-8176-4084-3},
	mrclass = {49-02 (35F20 49K20 49L20)},
	mrnumber = {2041617},
	mrreviewer = {Pierre Cardaliaguet},
	pages = {xiv+304},
	publisher = {Birkh{\"a}user Boston, Inc., Boston, MA},
	series = {Progress in Nonlinear Differential Equations and their Applications},
	title = {Semiconcave functions, {H}amilton-{J}acobi equations, and optimal control},
	volume = {58},
	year = {2004}}

@book{Evans_Gariepy_book,
	author = {Evans, Lawrence C. and Gariepy, Ronald F.},
	isbn = {0-8493-7157-0},
	mrclass = {28-02 (26-02 26Bxx 46E35)},
	mrnumber = {1158660},
	mrreviewer = {R. G. Bartle},
	pages = {viii+268},
	publisher = {CRC Press, Boca Raton, FL},
	series = {Studies in Advanced Mathematics},
	title = {Measure theory and fine properties of functions},
	year = {1992}}

@book{Hiriart-Urruty_Lemarechal_book2001,
	author = {Hiriart-Urruty, Jean-Baptiste and Lemar{\'e}chal, Claude},
	isbn = {3-540-42205-6},
	mrclass = {90-01 (26B25 49-01)},
	mrnumber = {1865628},
	pages = {x+259},
	publisher = {Springer-Verlag, Berlin},
	series = {Grundlehren Text Editions},
	title = {Fundamentals of convex analysis},
	url = {https://doi.org/10.1007/978-3-642-56468-0},
	year = {2001}}

@book{Katok_Hasselblatt_book,
	author = {Katok, Anatole and Hasselblatt, Boris},
	isbn = {0-521-34187-6},
	mrclass = {58Fxx (34Cxx 34Dxx 58-01 58F11 58F15)},
	mrnumber = {1326374},
	mrreviewer = {Edoh Amiran},
	note = {With a supplementary chapter by Katok and Leonardo Mendoza},
	pages = {xviii+802},
	publisher = {Cambridge University Press, Cambridge},
	series = {Encyclopedia of Mathematics and its Applications},
	title = {Introduction to the modern theory of dynamical systems},
	url = {https://doi.org/10.1017/CBO9780511809187},
	volume = {54},
	year = {1995}}

@article{Mather1991,
	author = {Mather, John N.},
	fjournal = {Mathematische Zeitschrift},
	issn = {0025-5874},
	journal = {Math. Z.},
	mrclass = {58F05 (58E30 58F11 58F22 70H35)},
	mrnumber = {1109661},
	number = {2},
	pages = {169--207},
	title = {Action minimizing invariant measures for positive definite {L}agrangian systems},
	url = {https://doi.org/10.1007/BF02571383},
	volume = {207},
	year = {1991}}

@article{Mather1993,
	author = {Mather, John N.},
	fjournal = {Universit{\'e} de Grenoble. Annales de l'Institut Fourier},
	issn = {0373-0956},
	journal = {Ann. Inst. Fourier (Grenoble)},
	mrclass = {58F05 (34C99 58E30 70H35)},
	mrnumber = {1275203},
	mrreviewer = {Gabriel P. Paternain},
	number = {5},
	pages = {1349--1386},
	title = {Variational construction of connecting orbits},
	url = {http://www.numdam.org/item?id=AIF_1993__43_5_1349_0},
	volume = {43},
	year = {1993}}

@InProceedings{10.1007/BFb0093512,
author="Aubry, Serge",
title="The devil's stair case transformation in incommensurate lattices",
booktitle="The Riemann Problem, Complete Integrability and Arithmetic Applications",
year="1982",
publisher="Springer Berlin Heidelberg",
pages="221--245",
isbn="978-3-540-39152-4"
}

@article {MR3622271,
    AUTHOR = {Klempnauer, Stefan and Schr\"oder, Jan Philipp},
     TITLE = {The stable norm on the 2-torus at irrational directions},
   JOURNAL = {Nonlinearity},
  FJOURNAL = {Nonlinearity},
    VOLUME = {30},
      YEAR = {2017},
    NUMBER = {3},
     PAGES = {912--942},
      ISSN = {0951-7715,1361-6544},
   MRCLASS = {37E45 (37D40 37J50)},
  MRNUMBER = {3622271},
MRREVIEWER = {M.\ Eugenia\ Rosado Mar\'ia},
       DOI = {10.1088/1361-6544/aa5520},
       URL = {https://doi.org/10.1088/1361-6544/aa5520},
}

@article {MR719634,
    AUTHOR = {Serge Aubry and Le Daeron, P. Y.},
     TITLE = {The discrete {F}renkel-{K}ontorova model and its extensions.
              {I}. {E}xact results for the ground-states},
   JOURNAL = {Phys. D},
  FJOURNAL = {Physica D. Nonlinear Phenomena},
    VOLUME = {8},
      YEAR = {1983},
    NUMBER = {3},
     PAGES = {381--422},
      ISSN = {0167-2789,1872-8022},
   MRCLASS = {58F05 (82A60)},
  MRNUMBER = {719634},
MRREVIEWER = {Dieter\ H.\ Mayer},
       DOI = {10.1016/0167-2789(83)90233-6},
       URL = {https://doi.org/10.1016/0167-2789(83)90233-6},
}

@article {MR719055,
    AUTHOR = {Aubry, Serge},
     TITLE = {The twist map, the extended {F}renkel-{K}ontorova model and the devil's staircase},
   JOURNAL = {Phys. D},
  FJOURNAL = {Physica D. Nonlinear Phenomena},
    VOLUME = {7},
      YEAR = {1983},
    NUMBER = {1-3},
     PAGES = {240--258},
      ISSN = {0167-2789,1872-8022},
   MRCLASS = {58F08 (58F05 82A55)},
  MRNUMBER = {719055},
MRREVIEWER = {Igor\ Gumowski},
       DOI = {10.1016/0167-2789(83)90129-X},
       URL = {https://doi.org/10.1016/0167-2789(83)90129-X},
}

@article{Aubry_1983exact,
doi = {10.1088/0022-3719/16/13/012},
url = {https://doi.org/10.1088/0022-3719/16/13/012},
year = {1983},
month = {may},
publisher = {},
volume = {16},
number = {13},
pages = {2497},
author = {Serge Aubry},
title = {Exact models with a complete Devil's staircase},
journal = {Journal of Physics C: Solid State Physics},
}

@article{aubry1983complete,
  title={Complete devil's staircase in the one-dimensional lattice gas},
  author={Serge Aubry},
  journal={Journal de Physique Lettres},
  volume={44},
  number={7},
  pages={247--250},
  year={1983},
  publisher={Les Editions de Physique}
}

@incollection {MR1219348,
    AUTHOR = {Serge Aubry},
     TITLE = {The concept of anti-integrability: definition, theorems and
              applications to the standard map},
 BOOKTITLE = {Twist mappings and their applications},
    SERIES = {IMA Vol. Math. Appl.},
    VOLUME = {44},
     PAGES = {7--54},
 PUBLISHER = {Springer, New York},
      YEAR = {1992},
      ISBN = {0-387-97858-5},
   MRCLASS = {58F30 (28D05 58F13)},
  MRNUMBER = {1219348},
MRREVIEWER = {M.\ L.\ Blank},
       DOI = {10.1007/978-1-4613-9257-6\_2},
       URL = {https://doi.org/10.1007/978-1-4613-9257-6_2},
}

@article{abe2024large,
  title={Large anomalous Hall effect in spin fluctuating devil’s staircase},
  author={Abe, Naoki and Hano, Yuya and Ishizuka, Hiroaki and Kozuka, Yusuke and Tadano, Terumasa and Tsujimoto, Yoshihiro and Yamaura, Kazunari and Ishiwata, Shintaro and Fujioka, Jun},
  journal={npj quantum materials},
  volume={9},
  number={1},
  pages={41},
  year={2024},
  publisher={Nature Publishing Group UK London}
}

@article {MR1384394,
    AUTHOR = {Bangert, Victor},
     TITLE = {Geodesic rays, {B}usemann functions and monotone twist maps},
   JOURNAL = {Calc. Var. Partial Differential Equations},
  FJOURNAL = {Calculus of Variations and Partial Differential Equations},
    VOLUME = {2},
      YEAR = {1994},
    NUMBER = {1},
     PAGES = {49--63},
      ISSN = {0944-2669,1432-0835},
   MRCLASS = {53C22 (58E10)},
  MRNUMBER = {1384394},
MRREVIEWER = {Edoh\ Amiran},
       DOI = {10.1007/BF01234315},
       URL = {https://doi.org/10.1007/BF01234315},
}

@article {MR668663,
    AUTHOR = {Bak, Per and Bruinsma, R.},
     TITLE = {One-dimensional {I}sing model and the complete devil's
              staircase},
   JOURNAL = {Phys. Rev. Lett.},
  FJOURNAL = {Physical Review Letters},
    VOLUME = {49},
      YEAR = {1982},
    NUMBER = {4},
     PAGES = {249--251},
      ISSN = {0031-9007},
   MRCLASS = {82A68},
  MRNUMBER = {668663},
       DOI = {10.1103/PhysRevLett.49.249},
       URL = {https://doi.org/10.1103/PhysRevLett.49.249},
}

@article {MR675056,
    AUTHOR = {Bak, Per},
     TITLE = {Commensurate phases, incommensurate phases and the devil's
              staircase},
   JOURNAL = {Rep. Progr. Phys.},
  FJOURNAL = {Reports on Progress in Physics},
    VOLUME = {45},
      YEAR = {1982},
    NUMBER = {6},
     PAGES = {587--629},
      ISSN = {0034-4885,1361-6633},
   MRCLASS = {82A25 (58E07 58F05 58F13 82-02)},
  MRNUMBER = {675056},
       URL = {http://stacks.iop.org/0034-4885/45/587},
}

@article {MR967638,
    AUTHOR = {Mather, John N.},
     TITLE = {Destruction of invariant circles},
   JOURNAL = {Ergodic Theory Dynam. Systems},
  FJOURNAL = {Ergodic Theory and Dynamical Systems},
    VOLUME = {8$^*$},
      YEAR = {1988},
     PAGES = {199--214},
      ISSN = {0143-3857,1469-4417},
   MRCLASS = {58F05 (34C27 58F08 58F27)},
  MRNUMBER = {967638},
MRREVIEWER = {Helmut\ R\"ussmann},
       DOI = {10.1017/S0143385700009421},
       URL = {https://doi.org/10.1017/S0143385700009421},
}

@incollection {MR945963,
    AUTHOR = {Bangert, Victor},
     TITLE = {Mather sets for twist maps and geodesics on tori},
 BOOKTITLE = {Dynamics reported, {V}ol.\ 1},
    SERIES = {Dynam. Report. Ser. Dynam. Systems Appl.},
    VOLUME = {1},
     PAGES = {1--56},
 PUBLISHER = {Wiley, Chichester},
      YEAR = {1988},
      ISBN = {3-519-02150-1},
   MRCLASS = {58F17 (58F05 82A68)},
  MRNUMBER = {945963},
MRREVIEWER = {Dietrich\ Flockerzi},
}

@incollection {MR1323222,
    AUTHOR = {Mather, John N. and Forni, Giovanni},
     TITLE = {Action minimizing orbits in {H}amiltonian systems},
 BOOKTITLE = {Transition to chaos in classical and quantum mechanics
              ({M}ontecatini {T}erme, 1991)},
    SERIES = {Lecture Notes in Math.},
    VOLUME = {1589},
     PAGES = {92--186},
 PUBLISHER = {Springer, Berlin},
      YEAR = {1994},
      ISBN = {3-540-58416-1},
   MRCLASS = {58F05 (58F27)},
  MRNUMBER = {1323222},
MRREVIEWER = {F.\ Cardin},
       DOI = {10.1007/BFb0074076},
       URL = {https://doi.org/10.1007/BFb0074076},
}

@book{mackay1985lectures,
  title={Lectures on orbits of minimal action for area-preserving maps},
  author={Mackay, Robert Sinclair and Stark, Jaroslav},
  year={1985},
  publisher={University of Warwick}
}

@article{rotondo2016devil,
  title={Devil’s staircase phase diagram of the fractional quantum Hall effect in the thin-torus limit},
  author={Rotondo, Pietro and Molinari, Luca Guido and Ratti, Piergiorgio and Gherardi, Marco},
  journal={Physical Review Letters},
  volume={116},
  number={25},
  pages={256803},
  year={2016},
  publisher={APS}
}

@article{ishimura1985devil,
  title={The Devil’s Staircase in an Ising Chain with Spin-Lattice Coupling},
  author={Ishimura, Norikazu},
  journal={Journal of the Physical Society of Japan},
  volume={54},
  number={3},
  pages={1131--1138},
  year={1985},
  publisher={The Physical Society of Japan}
}

@article{ruzzi2020hidden,
  title={Hidden Devil's staircase in a two-dimensional elastic model of spin crossover materials},
  author={Ruzzi, Gian and Cruddas, Jace and McKenzie, Ross H and Powell, Ben J},
  journal={arXiv preprint arXiv:2008.08738},
  year={2020}
}

@article {MR1374412,
    AUTHOR = {Graczyk, Jacek and \'Swi\c{a}tek, Grzegorz},
     TITLE = {Critical circle maps near bifurcation},
   JOURNAL = {Comm. Math. Phys.},
  FJOURNAL = {Communications in Mathematical Physics},
    VOLUME = {176},
      YEAR = {1996},
    NUMBER = {2},
     PAGES = {227--260},
      ISSN = {0010-3616,1432-0916},
   MRCLASS = {58F03 (58F14)},
  MRNUMBER = {1374412},
MRREVIEWER = {Andreas\ Stirnemann},
       URL = {http://projecteuclid.org/euclid.cmp/1104285997},
}

@article {MR863203,
    AUTHOR = {Moser, J\"urgen},
     TITLE = {Monotone twist mappings and the calculus of variations},
   JOURNAL = {Ergodic Theory Dynam. Systems},
  FJOURNAL = {Ergodic Theory and Dynamical Systems},
    VOLUME = {6},
      YEAR = {1986},
    NUMBER = {3},
     PAGES = {401--413},
      ISSN = {0143-3857,1469-4417},
   MRCLASS = {58F05 (49A10 49C05)},
  MRNUMBER = {863203},
MRREVIEWER = {Helmut\ R\"ussmann},
       DOI = {10.1017/S0143385700003588},
       URL = {https://doi.org/10.1017/S0143385700003588},
}

@article {MR2309170,
    AUTHOR = {Bernard, Patrick},
     TITLE = {Symplectic aspects of {M}ather theory},
   JOURNAL = {Duke Math. J.},
  FJOURNAL = {Duke Mathematical Journal},
    VOLUME = {136},
      YEAR = {2007},
    NUMBER = {3},
     PAGES = {401--420},
      ISSN = {0012-7094,1547-7398},
   MRCLASS = {37J50 (37J05)},
  MRNUMBER = {2309170},
MRREVIEWER = {Karl\ Friedrich\ Siburg},
       DOI = {10.1215/S0012-7094-07-13631-7},
       URL = {https://doi.org/10.1215/S0012-7094-07-13631-7},
}
\end{document}